\documentclass{elsart}
\usepackage{graphics}
\usepackage{graphicx}
\usepackage{epsfig}
\usepackage{amssymb,amsmath}
\usepackage[colorlinks]{hyperref}
\newtheorem{lemma}{Lemma}[section]
\newtheorem{corollary}{Corollary}[section]
\newtheorem{theorem}{Theorem}[section]
\newtheorem{proposition}{Proposition}[section]
\newtheorem{definition}{Definition}[section]
\newtheorem{remark}{Remark}[section]
\newtheorem{proof}{Proof}[section]
\begin{document}
\linespread{1}
\begin{frontmatter}
\title{A mixed multifractal analysis for quasi Ahlfors vector-valued measures}
\author{Adel Farhat}
\address{Labo. of Algebra, Number Theory and Nonlinear Analysis LR18ES50, Department of Mathematics, Faculty of Sciences, Monastir 5000, Tunisia.}
\ead{farhatadel1222@yahoo.fr}
\author{Anouar Ben Mabrouk\corauthref{cor1}\thanksref{label1}}
\address{Department of Mathematics, Higher Institute of Applied Mathematics and Computer Science, Street of Assad Ibn Alfourat, 3100 Kairouan, Tunisia.\\
Lab. Algebra, Number Theory and Nonlinear Analysis, UR11ES50, Department of Mathematics, Faculty of Sciences, 5000 Monastir, Tunisia.}
\ead{anouar.benmabrouk@fsm.rnu.tn}
\thanks[label1]{{Department of Mathematics, Faculty of Science, University of Tabuk, Saudi Arabia.}}
\corauth[cor1]{Corresponding author.}
\begin{abstract}
The multifractal formalism for measures in its original formulation is checked for special classes of measures such as doubling, self similar, and Gibbs-like ones. Out of these classes, suitable conditions should be taken into account to prove the validity of the multifractal formalism. In the present work, a large class of measures satisfying a weak condition known as quasi Ahlfors is considered in the framework of mixed multifractal analysis. A joint multifractal analysis of finitely many quasi Ahlfors probability measures is developed. Mixed variants of multifractal generalizations of Hausdorff and packing measures, and corresponding dimensions are introduced. By applying convexity arguments, some properties of these measures and dimensions are established. Finally, an associated multifractal formalism is introduced and proved to hold for the class of quasi Ahlfors measures. 
\end{abstract}
\end{frontmatter}
\maketitle
\section{Introduction}
The present work is devoted to the topic of multifractal analysis of measures and the validity of multifractal formalism. We precisely focus on the simultaneous behaviors of finitely many measures instead of a single measure as in the classical or original multifractal analysis of measures. We call such a study mixed multifractal analysis. Such a mixed analysis has been generating a great attention recently and thus proved to be powerful in describing the local behavior of measures especially fractal ones. (See \cite{BarreiraSaussol}, \cite{Bmabrouk1}, \cite{Bmabrouk3}, \cite{olsen1}, \cite{olsen2}, \cite{olsen2c}, \cite{olsen2d}, \cite{olsen2e}, \cite{Qu}, \cite{Xu-Wang}, \cite{Xu-Xu}, \cite{Xu-Xu-Zhong}, \cite{Yuan}, \cite{Zeng-Yuan-Xui}, \cite{Zhou-Feng}, \cite{Zhu-Zhou}).

In this paper, multi purposes will be done. Firstly we introduce the general context in which the new joint multifractal analysis will be developed. We introduce firstly some new variants of Hausdorff and packing measures based on some vector-valued measures controlled by a common measure assumed to be quasi-Ahlfors. Next, in the same section, associated multifractal dimensions due to the introduced measures are defined. Eventual relations between joint Hausdorff measures, joint packing measures, joint Hausdorff dimensions, joint packing dimensions as well as joint variants of Renyi dimensions are used. A second aim is to introduce a corresponding multifractal spectrum/formalism permitting to investigate the behavior of finitely many measures simultaneously. 

As it is noticed from the literature on multifractal analysis of measures, this latter always considered a single measure and studies its scaling behavior as well as the multifractal formalism associated. Recently, many works have been focused on the study of simultaneous behaviors of finitely many measures. In \cite{olsen3}, a mixed multifractal analysis has been developed dealing with a generalization of R\'enyi dimensions for a finite set of self similar measures, constituting thus a first motivation to our present paper. We intend to combine the generalized Hausdorff and packing measures, and the corresponding dimensions with Olsen's results in \cite{olsen2e} to define and develop a more general multifractal analysis for finitely many measures by studying their simultaneous regularity, spectrum and to define a mixed multifractal formalism which may describe better the geometry of the singularities's sets of these measures. We apply the techniques of L. Olsen especially in \cite{olsen1} and \cite{olsen2e} with the necessary modifications to study in details general mixed dimensions of simultaneously finite number of measures one of them at least is characterized by a quasi-Ahlfors property. Our results May be next shown to valid easily for a single measure, which confirms their generecity.

The assumption of being Ahlfors for one of the measures is essential contrarily to some existing works that have omitted such assumption and developed some questionable version of multifractal densities (\cite{Attiaetal,Cole,Douzi-Selmi}). Indeed, in such references the authors referred to \cite{Billingsley} to show that a real-valued dimension may exist without a link between the measure of a ball and its diameter. However, the authors did not pay attention to the fact that general probability measures (although being doubling) may not lead to multifractal dimensions. Indeed, it is assumed in \cite{Billingsley} (but not in \cite{Attiaetal} and \cite{Cole}) that 
\begin{itemize}
\item both the lower bound taken in the definition of the $\mu$-Hausdorff measure and the upper bound taken in the $\mu$-packing measure, may extend on all $\mu$-$\rho$-coverings and packings whenever $\mu$ is a Borel probability measure. Besides, a $\mu$-$\rho$-covering is already a covering by cylinders $C$ satisfying $\mu(C)<\rho$, where $\rho$ is the diameter or the radius.
\item The measure $\mu$ is assumed to be nonatomic. Otherwise, the set of $\mu$-$\rho$-coverings may be empty. 
\end{itemize}
Consequently, one may check for \cite{Attiaetal} and \cite{Cole} the validity of the multifractal formalism in a general case without any assumption on the measure of balls and their diameters. Relatively to \cite{olsen1}, where the measure $\nu$ is equivalent to a Lebesgue's one, the results obtained in \cite{Attiaetal} and \cite{Cole} are different. We proposed in the present work to relax hypothesis on the applied measures and assume a weak form of the so-called Alhfors measures. Backgrounds on such measures may be found in \cite{Edgar}, \cite{Jarvenpaaetal}, \cite{Mattila}, \cite{MattilaSaaramen}, \cite{Pajot}, \cite{Pesin-Weiss}.
\begin{definition}\label{quasiahlforsmeasure} Let $\nu$ be a Borel probability measure on $\mathbb{R}^d$. We say that $\nu$ is a quasi-Ahlfors measure with regularity index $\alpha>0$ if it satisfies
$$
\displaystyle\limsup_{|U|\longrightarrow0}\displaystyle\frac{\mu(U)}{|U|^{\alpha}}<+\infty.
$$
We say that $\nu$ is an Ahlfors measure with regularity index $\alpha>0$ if it satisfies
$$
0<\displaystyle\liminf_{|U|\longrightarrow0}\displaystyle\frac{\mu(U)}{|U|^{\alpha}}\leq\displaystyle\limsup_{|U|\longrightarrow0}\displaystyle\frac{\mu(U)}{|U|^{\alpha}}<+\infty.
$$
In the rest of the paper, we will write $\mathcal{QAHP}(\mathbb{R}^{d})$ (respectively, $\mathcal{AHP}(\mathbb{R}^{d})$) to designate the set of quasi-Ahlfors (respectively, Ahlfors) probability measures on $\mathbb{R}^d$.
\end{definition}
With these assumptions, it becomes natural to extend the multifractal analysis to the cases introduced in \cite{Attiaetal} and \cite{Cole}. The readers may refer to \cite{Billingsley}, \cite{Das}, \cite{Menceuretal}, \cite{olsen1}, \cite{olsen2}, \cite{Qu}, \cite{Xu-Wang}, \cite{Xu-Xu}, \cite{Xu-Xu-Zhong}, \cite{Ye1}, \cite{Ye1}, \cite{Yuan}, \cite{Zeng-Yuan-Xui}, \cite{Zhou-Feng}, \cite{Zhu-Zhou} for more understanding.

Resuming, mixed multifractal analysis is a natural extension of multifractal analysis of single objects such as measures, functions, statistical data, distributions, ... It is developed quite recently (since 2014) in the pure mathematical point of view. In physics and statistics, it was appearing on different forms but not really and strongly linked to the mathematical theory. See for example \cite{Ganetal}, \cite{Maneveauetal}. In many applications such as clustering topics, each attribute in a data sample may be described by more than one type of measure. This leads researchers to apply measures well adopted for mixed-type data. See for example \cite{Ganetal}.  

In section 2, we address a brief review on single/mixed multifractal analysis. We precisely focus on eventual applications, existing works and possible future direction, already with enlightening the link to our work in the present paper and the novelty relatively to existing works in the same topic. In section 3, we introduce the new context of the joint multifractal analysis. We introduce the new joint extensions of Hausdorff and packing measures and we show that they associate as for the classical case some multifractal dimensions to sets. We prove some properties of convexity, monotony, upper and lower bounds for these dimensions. Next, section 4 is devoted to the development of an associated multifractal spectrum relative to the new dimensions and measures developed previously. An upper bound of such a spectrum is provided. In section 5, we develop an associated multifractal formalism issued from the previous framework. Results on the validity of such formalism are established as well as an application to the case of projections of measures. Section 6 is a conclusion, in which we  resume the results obtained, and raise eventual extensions and applications.
\section{On the utility of mixed multifractal analysis and motivations}
In a simple description, mixed multifractal analysis may be defined as mathematical glasses that permit to capture and/or to quantify the transient higher-order dependence beyond correlation of many measures, functions, time series, and distributions.

Multifractal analysis has now applications and/or links to many cases in AI. It is indeed used in smart grids such as fractal ones, smart cities, buildings, nano-materials such as nano-circuits, patterns, complexity in nature, Patterns, biosignals, bioimages, ...etc. The bibliography in such topics is growing up rapidly. Interested readers may refer to the further reading section (\cite{Cao,Cattani2007,Chandrasekhar,Dauphine,Ghanbarian,Ghosh,Janahmadov,Nakayama,Ouadfeul,Seuront,Tolotti,ZhangBook}) for more information.

Mixed multifractal analysis (MXMA) has been applied in explaining joint movements in volatility for asset markets such as joint multifractal Markov-switching models. Besides, mixed multifractal analysis is not really new in financial series processing. It has been in contrast merged under the name of multivariate multifractal analysis (MVMA), where many situations in financial markets, and their volatility have been described. Multivariate models have been also applied for long memory with  mixture distributions. See for instance \cite{Calvet-Fisher,Drozdz,Fan,Liu,Liu1}.

Mixed multifractal analysis of measures has been applied in \cite{Dai} to extract the properties of stock index series, and in exploiting the eventual inner relationship that relates them. 

Mixed multifractal analysis has been also applied in the last decade to generalize the so-called detrended fluctuation analysis, which is applied widely for modeling time series. Mixed multifractal analysis has been emerged in such a field to model the nonlinear correlation in time series and/or their volatility to understand better the dependence between markets' indices. See \cite{BenMabroukDFA,Biswas,Jiang,Manshour}.

Recently, joint multifractal analysis, another face of mixed multifractal analysis, has been applied in economic-environmental linkage of indices. In \cite{Bozkus}, the dynamics of atmospheric carbon emissions and industrial production index have been investigated by means of joint multifractal analysis for both short and long terms. Already in the study of environmental indices and/or climate factors, a mixed multifractal study has been developed in \cite{Chu} to understand the thermal structure of the mixed layer in sea surface. Recall that temperature changes are basic causes of the well-known NINO phenomenon detected on sea surfaces. 

In \cite{Wendt}, multivariate multifractal analysis has been applied to texture characterization of natural images. The MVMA shows that image intensity may be subject to fluctuated regularity. In \cite{Wang}, MVMA has been applied for characterizing volatility and cross-correlation for agricultural futures' markets. Several statistical measures such as autocorrelation, coupling correlation as well as cross-correlation for some Chicago board of trade series have been estimated via the MVMA to decide about the state of markets.

In relation to AI applications of fractals and multifractals, Castiglioni and Faini addressed in \cite{Castiglioni-Faini} a multivariate multifractal analysis for biomedical applications by developing a fast DFA algorithm on maximally overlapped blocks suitable for short and long time series due to physiological applications. See also \cite{Wang1}. 

From a theoretical point of view, in \cite{DouziERA}, a mixed multifractal analysis has been developed to detect mutual singularities of the multifractal Hausdorff measure $\mathcal{H}_{\mu}^{q,t}$ and the packing measure $\mathcal{P}_{\mu}^{q,t}$. This differs from our work as we seek here the mutual singularities for the measure $\mu$ included in the definition of the Hausdorff/packing measure and not the singularities of $\mathcal{H}_{\mu}^{q,t}$ and $\mathcal{P}_{\mu}^{q,t}$. 

In \cite{Attia-Selmi1}, a mixed multifractal analysis has been developed for box dimension of measures. Recall that box dimensions may coincide with the Hausdorff and packing ones under suitable conditions such as the compactness of the metric space. In our work, we get such coincidence without restricting to compact metric spaces. See also \cite{DouziSamtiSelmi,SelmiProy2021}. In \cite{DouziATA}, the authors applied the mixed multifractal analysis of measures to estimate the R\'enyi dimension relative to a couple of measures where the product of such measures is controlled by the diameter/radius of the coverings. In \cite{SelmiActa}, some known concepts about the multifractal analysis are revisited in the framework of a relative multifractal analysis. New simplified proofs and examples have been provided.

More about the use of mixed multifractal analysis of measures as well as functions and time series or images may be found in \cite{Abry,Avishek,Dai1,Kinnison,Oral} and the references therein. 
\section{New joint multifractal measures and associated dimensions} 
In this section, we propose to develop the general framework of the joint multifractal analysis. General joint variants of both Hausdorff and packing measures will be introduced as well as the associated dimensions of sets.

Let $\mathcal{P}(\mathbb{R}^{d})$ be the set of all probability Borel measures on $\mathbb{R}^{d}$. For a single or vector-valued measure $\mu$ denote $S_\mu$ its topological support. Let $k\in\mathbb{N}$ be fixed and $\mu=(\mu_{1},\mu_{2},...,\mu_{k})\in\mathcal{P}(\mathbb{R}^{d})^{k}$. Let also $q=(q_{1},q_{2},...,q_{k})\in\mathbb{R}^{k}$, $x\in\mathbb{R}^d$ and $r>0$. We denote
$$
\left[\mu(B(x,r))\right]^{q}=\left[ \mu_{1}(B(x,r))\right] ^{q_{1}}\times
...\times \left[ \mu_{k}(B(x,r))\right] ^{q_{k}},
$$
where we designate by $B(x,r)$ the ball of center $x$ and radius $r$. Next, given $E\subseteq\mathbb{R}^d$ and $\varepsilon>0$, we call an $\varepsilon$-covering of $E$ any finite or countable set ($U_{i})_{i}$ of non-empty subsets $U_{i}\subseteq\mathbb{R}^d$ satisfying
\begin{equation}\label{epsiloncovering}
 E\subseteq{\cup}_iU_{i}\hbox{ and }|U_{i}|<\varepsilon,
\end{equation}
where $|.|$ is the diameter. 

The last assumption is already assumed in \cite{Billingsley}, \cite{Bmabrouk1}, \cite{Bmabrouk3}, \cite{Menceuretal}, \cite{olsen1}, \cite{olsen2}, \cite{olsen2b}, \cite{olsen2c}, \cite{olsen2d}, \cite{olsen3} but unfortunately not in \cite{Attiaetal} and \cite{Cole}. When assuming that the measure is quasi-Ahlfors, this assumption is not necessary and may be replaced by the original one in \cite{Billingsley} on $\mu$-$\varepsilon$-coverings.
\begin{definition} Given $E\subseteq\mathbb{R}^{d}$ and $\varepsilon >0$, we call an $\mu$-$\varepsilon$-covering of $E$ any at most countable set $(U_{i})_{i}$ of non-empty subsets $U_{i}\subseteq\mathbb{R}^{d}$ satisfying
$$
E\subseteq{\cup}_iU_{i}\hbox{ and }\mu(U_{i})<\varepsilon .
$$
For a subset $E\subseteq\mathbb{R}^{d}$ and $\varepsilon>0$, a countable set $(B(x_{i},r_{i}))_{i}$ of balls in $\mathbb{R}^{d}$ is called a centered $\varepsilon$-$\mu$-covering of $E$ if
$$
E\subseteq \underset{i}{\cup }B(x_{i},r_{i}),\text{ }x_{i}\in E\text{ and }%
0<\mu (B(x_{i},r_{i}))<\varepsilon \text{ for all }i.
$$
\end{definition}
\begin{definition} Let $\mu=(\mu_{1},\mu_{2},...,\mu_{k})$ be a positive vector-valued measure on $E=E_{1}\times E_{2}\times ...\times E_{k}$. Then the support of $\mu$ is the complement of the largest open subset $U$ of $E$ such that $\mu(U)=(0,0,...,0)$.
\end{definition}
\begin{theorem}\textbf{(Besicovitch covering theorem)}
There exists a fixed $\mathcal{N}_B\in\mathbb{N}$ satisfying the following assertion: For any $A\subset\mathbb{R}^{d}$ and any bounded set $(r_{x})_{x\in A}\subset(0,\infty)$, there exists ${N}_B$ subsets $C_{1},C_{2},...,C_{N_B}$ of $\left\{B(x,r_{x}),\,x\in A\right\}$, such that
\begin{enumerate}
\item Each set $C_{i}$ is finite or countable, for all $i=1,...,N_B$.
\item Each $C_{i}$ is composed of disjoint sets, for all $i=1,...,{N}_B$.
\item $A\subseteq\displaystyle\bigcup_{i=1}^{N_B}\displaystyle\bigcup_{B\in C_i}B$.
\end{enumerate}
\end{theorem}
We now proceed in introducing the construction of the new variants of multifractal Hausdorff and packing measures and the associated dimensions. We will see later that being quasi-Ahlfors is necessary for at least one measure.

Let $\xi=(\mu,\nu)=(\mu_{1},\mu_{2},...,\mu_{k},\nu)\in\mathcal{P}(\mathbb{R}^{d})^{k}\times\mathcal{QAHP}(\mathbb{R}^d)$ and $(q,t)=(q_{1},q_{2},...,q_{k},t)\in\mathbb{R}^{k+1}$. For $E\subset\mathbb{R}^d$ and $\varepsilon >0$, let 
$$
{\overline{\mathcal{H}}}_{\xi,\varepsilon}^{q,t}(E)=\inf\biggl\{\sum_i(\mu(B(x_{i},r_{i})))^{q}(\nu(B(x_{i},r_{i})))^{t}\biggr\}
$$
and
$$
{\overline{\mathcal{H}}}_{\xi}^{q,t}(E)=\lim_{\varepsilon\downarrow0}{\overline{\mathcal{H}}}_{\xi,\varepsilon}^{q,t}(E),
$$
where the inf is over the set of centred $\varepsilon$-coverings of $E$. Similarly, let 
$$
{\overline{\mathcal{P}}}_{\xi,\varepsilon}^{q,t}(E)=\sup\{\sum_i(\mu(B(x_{i},r_{i})))^{q}(\nu(B(x_{i},r_{i})))^{t}\}
$$
and
$$
{\overline{\mathcal{P}}}_{\xi}^{q,t}(E)=\lim_{\varepsilon\downarrow0}{\overline{\mathcal{P}}}_{\xi,\varepsilon}^{q,t}(E),
$$
where the sup is over the set of centred $\varepsilon$-packings of $E$. 
\begin{definition}\label{mixedhausdorffandpackinggeneralisationsmeasures}
The mixed generalized Hausdorff measure relatively to $\xi$ is stated as 
$$
\mathcal{H}_{\xi}^{q,t}(E)=\displaystyle\sup_{F\subseteq E}{\overline{\mathcal{H}}}_{\xi}^{q,t}(F).
$$
The mixed generalized packing measure relatively to $\xi$ is expressed by
$$
\mathcal{P}_{\xi}^{q,t}(E)=\inf_{E\subseteq\cup_iE_{i}}{\sum_i}{\overline{\mathcal{P}}}_{\xi}^{q,t}(E_{i}).
$$
\end{definition}
It is straightforward that $\mathcal{H}_{\xi}^{q,t}$ and $\mathcal{P}_{\xi}^{q,t}$ are outer metric and regular measures on $\mathbb{R}^d$. Borel sets are thus measurable relatively to them. Furthermore, we may prove using the well known Besicovitch covering theorem that
\begin{equation}\label{HandPborneeslunparlautre}
\mathcal{H}_{\xi}^{q,t}(E)\leq\mathcal{N}_B\mathcal{P}_{\xi}^{q,t}\leq\mathcal{N}_B{\overline{\mathcal{P}}}_{\xi}^{q,t}(E),\;\forall(q,t)\in\mathbb{R}^{k+1},\;\forall E\subseteq \mathbb{R}^{d}.
\end{equation}
$\mathcal{N}_B$ is the constant number related to the Besicovitch covering theorem.

The original idea in \cite{olsen1} considers for $\mu\in\mathcal{P}(\mathbb{R}^d)$ and $(q,t)\in\mathbb{R}^2$ the dimension function
$$
h_{q,t}(r)=\mu(B(x, r))^qr^t,\;\; r > 0,
$$
and associates variants $\mathcal{H}_{h_{q,t}}$ for Hausdorff and $\mathcal{P}_{h_{q,t}}$ for packing measures relatively to $h_{q,t}$. These measures have special supports 
$$
X_\mu(\alpha)=\Biggl\{x\in\mathcal{S}_\mu,\;\displaystyle\lim_{r\downarrow0}\displaystyle\frac{\log(\mu(B(x,r)))}{\log\,r}=\alpha\Biggr\}
$$
known as multifractal decomposition sets, for suitable parameters $\alpha$ and $\mu$ as in doubling, H\"olderian, and Gibbs cases. To show the validity of the multifractal formalism in some special cases $X_\nu(f(\alpha))$ for suitable functions $f$, a large deviation formalism has been applied to obtain Gibbs measures supported by these sets. Billingsley theorem known in ergodic theory has been next applied to deduce the dimension of $X_\mu(\alpha)$ from the dimension of $X_\nu(f(\alpha))$.

In the present situation, a more general case is investigated for vector-valued measures using the cross-correlation dimension function 
$$
H_{q_1,\dots,q_k,t}(r)=\mu_1(B(x,r))^{q_1}\dots\mu_k(B(x,r))^{q_k}\nu(B(x,r))^{t},\;\; r > 0 ,
$$
for $q_1,\dots,q_k,t\in\mathbb{R}$, and where $\mu_1,\dots,\mu_k,\nu$ are Borel probability measures on $\mathbb{R}^d$. $\nu$ plays the role of a gauge function used to control the simultaneous behavior of the measures $\mu_i$, $i=1,\dots,k$. Variants of Hausdorff and packing measures relatively to the multi-variables dimension function $H_{q_1,\dots,q_k,t}$ are introduced. For example, in the single case $k=1$ and the choice $\nu(B(x,r))\simeq r$, we immediately obtain 
$$
H_{q,t}(r) = h_{q,t}(r)
$$
for all $q,t\in\mathbb{R}$, which means that the variants of Hausdorff and packing measures relative to $H_{q_1,\dots,q_k,t}$ are indeed extensions of those introduced in \cite{olsen1}. 

The results developed in the present work lie in the whole topic of dimension theory, especially Carath\'eodory dimension, which in turns constitutes a general form of Hausdorff one. Indeed, in \cite{Pesin-Weiss}, a link between dynamical systems and dimension theory has been pointed out, which makes a motivation to our work here. The idea in \cite{Pesin-Weiss} consists in interpreting the concept of dimension as a C-structure. Mathematically speaking, let $\mathcal{C}$ be any collection of subsets of a space $X$ known as the physical space, and consider some non-negative functions $\varphi$, $\Lambda$ and $\Theta$ defined on $\mathcal{C}$, which play respectively the role of the diameter and some statistical physics measures. In this framework, the dimension function $\mathcal{H}_{\xi}^{q,t}$ applied in the present work may be interpreted as a multivariate form of the Carath\'eodory one. Under suitable conditions on these functions, we may introduce a free energy analogue to the fractal measures by setting
$$
\mathcal{H}_\alpha(X)= \inf\{\displaystyle\sum_{U\in\mathcal{G}}\varphi(U)\Lambda(U)^\alpha\},
$$
where $\mathcal{G}$ depends eventually on $X$ and $\Lambda$. This which yields as usual a dimension called the capacity of $X$. One of the related frameworks in multifractal analysis is the adoption of the continuous analysis to discrete spaces such as $\mathbb{Z}^d$. Some essays have been already started the last decades by comparing the log-measure of balls to some quantities that are not log-powers of the diameter. As we know, the Hausdorff dimension of discrete sets is not important in the framework of classical definitions. Therefore, different variants of function dimensions should be applied appropriately for discrete sets. These set functions are no longer power-laws of the diameter. Interested readers may refer to \cite{Aversa-Bandt}, \cite{Barlow-Taylor}, \cite{Glasscock} and the references therein for the discrete framework.

In another parallel direction, J. Cole in \cite{Cole} proposed to control the analyzed measure $\mu$ by another suitable measure $\nu$ via a relative multifractal analysis of the relative singularity sets 
$$
X_{\xi}(\alpha,\beta)=\Biggl\{x\in\mathcal{S}_{\xi},\;\displaystyle\lim_{r\downarrow0}\displaystyle\frac{\log(\mu(B(x,r)))}{\log\,r}=\alpha\,\mbox{and}\,
\displaystyle\lim_{r\downarrow0}\displaystyle\frac{\log(\mu(B(x,r)))}{\log(\nu(B(x,r)))}=\beta\Biggr\}.
$$
We now introduce the associated mixed dimensions relative to the generalized Hausdorff and packing measures $\mathcal{H}_{\xi}^{q,t}$ and $\mathcal{P}_{\xi}^{q,t}$. We will notice the necessity of the quasi-Ahlfors assumption. We have the following result which prepares to introduce the new mixed dimensions and shows the necessity of being Ahlfors for some measures composing the $(k+1)$-tuple measure $\xi$.
\begin{lemma}\label{coupures-1}
Let $\xi=(\mu,\nu)\in\mathcal{P}(\mathbb{R}^d)^k\times\mathcal{QAHP}(\mathbb{R}^d)$ and $E\subseteq\mathbb{R}^d$. $\forall q\in\mathbb{R}^{k}$, The set $\Gamma_{q}:=\left\{t\in\mathbb{R},\;\mathcal{H}_{\xi}^{q,t}(E)<+\infty\right\}$ is nonempty. 
\end{lemma}
\begin{proof} Let $\alpha,M\in\mathbb{R}_{+}$ be such that
$$
\displaystyle\limsup_{|U|\rightarrow0}\displaystyle\frac{\nu(U)}{|U|^{\alpha}}<M.
$$
We obtain, for some $\delta>0$ and $\forall r$, $0<r<\delta$, 
$$
\nu (U)\leq M|U|^{\alpha};\;\forall U,\;|U|<r.
$$
Consider next a $\delta$-covering $(B(x_{i},r_{i}))_{i}$ of $E$, and the $\mathcal{N}_B$ Besicovitch collections. We obtain 
$$
\sum_i\mu (B(x_{i},r_{i}))^{q}\nu (B(x_{i},r_{i}))^{t}\leq 
\sum_{i=1}^{\mathcal{N}_B}\sum_j\mu(B(x_{ij},r_{ij}))^{q}\nu (B(x_{ij},r_{ij}))^{t}.
$$
Whenever $q\geq0$, the right hand term is bounded by
$$
\sum_{i=1}^{\mathcal{N}_B}\sum_j\nu(B(x_{ij},r_{ij}))^{t}.
$$
For $t=1,$ this becomes
$$
\sum_{i=1}^{\mathcal{N}_B}\sum_j\nu(B(x_{ij},r_{ij})).
$$
As the $(B(x_{ij},r_{ij}))_{j}$ are disjoint, the last quantity will be bounded by
$$
\sum_{i=1}^{\mathcal{N}_B}\nu\left(\cup_jB(x_{ij},r_{ij})\right)\leq\mathcal{N}_B\nu(\mathbb{R}^{d})=\mathcal{N}_B.
$$
Consequently, we obtain
$$
\mathcal{H}_{\xi}^{q,1}(E)<+\infty .
$$
Assume now that there exist $i_{j},$ $1\leq j\leq k$ such that $q_{i_{j}}\leq 0$. For $t>0,$ we get
$$
\nu (B(x_{i_{j}},r_{i_{j}}))^{t}\leq M^{t}r_{i_{j}}^{\alpha t},\forall j.
$$
As a result, we get
$$
\sum_j\mu(B(x_{i_{j}},r_{i_{j}}))^{q}\nu(B(x_{i_{j}},r_{i_{j}}))^{t}\leq2^{-\alpha t}M^{t}\sum_j\mu(B(x_{i_{j}},r_{i_{j}}))^{q}(2r_{i_{j}})^{\alpha t}.
$$
Let next $t>\frac{1}{\alpha}\left[\max\left(1,dim_{\mu}^{q}(E)\right)\right]$. We obtain
$$
\mathcal{H}_{\xi}^{q,t}(E)\leq2^{-\alpha t}M^{t}\mathcal{H}_{\xi}^{q,\alpha t}(E)<+\infty.
$$
\end{proof}
As a consequence of Lemma \ref{coupures-1}, we get 
\begin{equation}\label{coupures-2a}
\mathcal{H}_{\xi}^{q,t}(E)<+\infty\;\;\Longrightarrow\;\;\mathcal{H}_{\xi}^{q,s}(E)=0,\;\;\forall\,s>t,
\end{equation}
and
\begin{equation}\label{coupures-2b}
\mathcal{H}_{\xi}^{q,t}(E)>0\;\;\Longrightarrow\;\;\mathcal{H}_{\xi}^{q,s}(E)=+\infty,\;\;\forall\,s<t.
\end{equation}
This permits to introduce now the generalised mixed multifractal dimensions due to the variants $\mathcal{H}_{\xi}^{q,t}$ and $\mathcal{P}_{\xi}^{q,t}$.
\begin{proposition}\label{existencecoupures} For any set $E\subseteq\mathbb{R}^d$, there exists unique values denoted by  $dim_{\xi}^{q}({E})$, $\Delta_{\xi}^{q}({E})$
 and $Dim_{\xi}^{q}({E})$ in $\left[-\infty,+\infty\right]$, and satisfying respectively,
\begin{enumerate}
\item 
$$
\mathcal{H}_{\xi}^{q,t}(E)=\left\{ 
\begin{array}{c}
\infty\hbox{ for }t<dim_{\xi}^{q}({E}), \\ 
0\hbox{ for }t>dim_{\xi}^{q}({E}).
\end{array}
\right.
$$
\item
$$
{\overline{\mathcal{P}}}_{\xi}^{q,t}(E)=\left\{
\begin{array}{c}
\infty\hbox{ for }t<\Delta_{\xi}^{q}({E}),\\ 
0\hbox{ for }t>\Delta_{\xi}^{q}({E}).
\end{array}
\right. 
$$
\item
$$
\mathcal{P}_{\xi}^{q,t}(E)=\left\{ 
\begin{array}{c}
\infty\hbox{ for }t<Dim_{\xi}^{q}({E}),\\ 
0\hbox{ for }t>Dim_{\xi}^{q}({E}).
\end{array}
\right. 
$$
\end{enumerate}
\end{proposition}
\begin{proof} 
Item 1 is a consequence of equations (\ref{coupures-2a}) and (\ref{coupures-2b}) by setting
$$
dim_{\xi}^{q}({E})=\inf\{\,t\in\mathbb{R}\,;\;\mathcal{H}_{\xi}^{q,t}(E)=0\}.
$$ 
Item 2 is a consequence of the following assertion, stating that for all $E\subseteq\mathbb{R}^d$ and $t>0$,
$$
{\overline{\mathcal{P}}}_{\xi}^{q,t}(E)<+\infty\;\;\Longrightarrow\;\;{\overline{\mathcal{P}}}_{\xi}^{q,s}(E)=0,\;\;\forall\,s>t,
$$
and
$$
{\overline{\mathcal{P}}}_{\xi}^{q,t}(E)>0\;\;\Longrightarrow\;\;{\overline{\mathcal{P}}}_{\xi}^{q,s}(E)=+\infty,\;\;\forall\,s<t,
$$
by setting
$$
\Delta_{\xi}^{q}({E})=\inf\{\,t\in\mathbb{R}\,;\;{\overline{\mathcal{P}}}_{\xi}^{q,t}(E)=0\}.
$$
Item 3 is a consequence of the following assertion, stating that for all $E\subseteq\mathbb{R}^d$ and $t>0$,
$$
{{\mathcal{P}}}_{\xi}^{q,t}(E)<+\infty\;\;\Longrightarrow\;\;{{\mathcal{P}}}_{\xi}^{q,s}(E)=0,\;\;\forall\,s>t,
$$
and
$$
{{\mathcal{P}}}_{\xi}^{q,t}(E)>0\;\;\Longrightarrow\;\;{{\mathcal{P}}}_{\xi}^{q,s}(E)=+\infty,\;\;\forall\,s<t,
$$ 
by setting
$$
Dim_{\xi}^{q}({E})=\inf\{\,t\in\mathbb{R}\,;\;\mathcal{P}_{\xi}^{q,t}(E)=0\}.
$$
\end{proof}
\begin{definition} 
\begin{itemize}
\item $dim_{\xi}^{q}({E})$ is called the mixed multifractal generalization of the Hausdorff dimension of the set $E$.
\item $Dim_{\xi}^{q}({E})$ is called the mixed multifractal generalization of the packing dimension of the set $E$.
\item $\Delta_{\xi}^{q}({E})$ is called the mixed multifractal generalization of the logarithmic index of the set $E$.
\end{itemize}
\end{definition}
Remark easily that the original definitions of the single Hausdorff and packing measures and dimensions are obtained for $k=1$ and $q=0$. Besides, the multifractal generalizations due to Olsen are obtained for $k=1$ and $q\in\mathbb{R}$. We have precisely,
$$
dim_{\xi}^{Q_{i}}({E})=dim_{\mu_{i},\nu}^{q_{i}}({E}),\;
Dim_{\xi}^{Q_{i}}({E})=Dim_{\mu_{i},\nu}^{q_{i}}({E}),\;
\Delta_{\xi}^{Q_{i}}({E})=\Delta_{\mu_{i},\nu}^{q_{i}}({E}),
$$
and
$$
dim_{\xi}^{0}({E})=dim_{\nu}({E}),\;
Dim_{\xi}^{0}({E})=Dim_{\nu}({E}),\;
\Delta_{\xi}^{0}({E})=\Delta_{\nu}({E}).
$$
From now on, we will denote for $E\subseteq\mathbb{R}^{d}$, $q=(q_{1},q_{2},...,q_{k})\in\mathbb{R}^{k}$, $t\in\mathbb{R}$, $\mu\in\mathcal{P}(\mathbb{R}^d)$ and $\nu\in\mathcal{QAHP}(\mathbb{R}^d)$, 
$$
b_{\xi}(E,q)=dim_{\xi}^{q}({E}),\;B_{\xi}(E,q)=Dim_{\xi}^{q}({E}),\;\Delta_{\xi}(E,q)=\Lambda_{\xi}^{q}({E}).
$$ 
For $E=S_{(\mu,\nu)}$, we denote 
$$
b_{\xi}(q)=dim_{\xi}^{q}(S_{(\mu,\nu)}),\;B_{\xi}(q)=Dim_{\xi}^{q}(S_{(\mu,\nu)}),\;\Delta_{\xi}(q)=\Lambda_{\xi}^{q}(S_{(\mu,\nu)}).
$$
For $x=(x_1,x_2,\dots,x_k)$ and $q=(q_1,q_2,\dots,q_k)$ in $\mathbb{R}^k$ we denote 
$$
|x|=x_1+x_2+\dots+x_k\;\;\mbox{and}\;\;x^q=x_1^{q_1}x_2^{q_2}\dots x_k^{q_k}.
$$
\begin{theorem} The following assertions are true.
\begin{description}
\item\textbf{a.} $b_{\xi}(.,q)$ and $B_{\xi}(.,q)$ and $\Delta_{\xi}(.,q)$ are non-decreasing with respect to the inclusion proprerty in $\mathbb{R}^{d}$.
\item\textbf{b.} $b_{\xi}(.,q)$ and $B_{\xi}(.,q)$ are $\sigma$-stable.
\item\textbf{c.} $B_{\xi}(q)$ and $\Lambda_{\xi}(q)$ are convex.
\item\textbf{d.} For $\widehat{q_{i}}=(q_{1},...,q_{i-1},q_{i+1},...,q_{k})$ fixed, the functions $q_{i}\mapsto
b_{\xi}(q)$, $q_{i}\mapsto B_{\xi}(q)$ and $q_{i}\mapsto\Lambda_{\xi}(q)$ are non-increasing, $\forall\,i=1,2,...,k$.
\end{description}
\end{theorem}
\begin{proof}
\textbf{a.} follows from the non decreasing property of $\mathcal{H}_{\xi}^{q,t}$, $\mathcal{P}_{\xi}^{q,t}$ and $\overline{\mathcal{P}}_{\xi}^{q,t}$ with respect to the inclusion in $\mathbb{R}^{d}$.\\
\textbf{b.} follows from the sub-additivity property of $\mathcal{H}_{\xi}^{q,t}$ and $\mathcal{P}_{\xi}^{q,t}$ in $\mathbb{R}^{d}$.\\
\textbf{c.} We start by proving the convexity of $\Lambda_{\xi}(E,.)$. Consider $p,q\in\mathbb{R}^{k}$, $\alpha\in]0,1[$, and $s,t\in\mathbb{R}$, such that
$$
s>\Lambda_{\xi}(E,p)\mbox{ and }t>\Lambda_{\xi}(E,q).
$$
Let $\varepsilon>0$ be fixed arbitrary, and $(B_{i}=B(x_{i},r_{i}))_{i}$ be a centered $\varepsilon$-packing of $E$. We have
\begin{eqnarray*}
&&{\sum_i}(\mu(B_{i}))^{\alpha q+(1-\alpha)p}(\nu(B_{i}))^{\alpha t+(1-\alpha)s}\\
&\leq&\left[{\sum_i}(\mu(B_{i}))^{q}(\nu(B_{i}))^{t}\right]^{\alpha }\left[{\sum_i}(\mu(B_{i}))^{p}(\nu(B_{i}))^{s}\right]^{1-\alpha}.
\end{eqnarray*}
Hence,
$$
{\overline{\mathcal{P}}}_{\xi,\varepsilon}^{^{\alpha q+(1-\alpha)p,\alpha t+(1-\alpha)s}}(E)\leq({\overline{\mathcal{P}}}_{\xi,\varepsilon}^{q,t}(E))^{\alpha}({\overline{\mathcal{P}}}_{\xi,\varepsilon}^{p,s}(E))^{1-\alpha}.
$$
The limit as $\varepsilon\downarrow 0$ gives 
$$
{\overline{\mathcal{P}}}_{\xi}^{^{\alpha q+(1-\alpha)p,\alpha t+(1-\alpha)s}}(E)\leq({\overline{\mathcal{P}}}_{\xi}^{q,t}(E))^{\alpha}({\overline{\mathcal{P}}}_{\xi}^{p,s}(E))^{1-\alpha}.
$$
Consequently,
$$
{\overline{\mathcal{P}}}_{\xi}^{^{\alpha q+(1-\alpha)p,\alpha t+(1-\alpha)s}}(E)=0,\;\forall\,s>\Lambda_{\xi}(E,p)\hbox{ and }t>\Lambda_{\xi}(E,q).
$$
It results that 
$$
\Lambda_{\xi}(\alpha q+(1-\alpha)p,E)\leq\alpha\Lambda_{\xi}(E,q)+(1-\alpha)\Lambda_{\xi}(E,p).
$$
We now prove the convexity of $B_{\xi}(E,.)$. We set in this case%
$$
t=B_{\xi}(E,q)\hbox{ and }s=B_{\xi}(E,p).
$$
We have 
$$
\mathcal{P}_{\xi}^{q,t+\varepsilon}(E)=\mathcal{P}_{\xi}^{q,s+\varepsilon}(E)=0.
$$
Therefore, there exists $(H_{i})_{i}$ and $(K_{i})_{i}$ coverings of the set $E$ for which
$$
{\sum_i}{\overline{\mathcal{P}}}_{\xi}^{q,t+\varepsilon}(H_{i})\leq C<+\infty\hbox{ and }{\sum_i}{\overline{\mathcal{P}}}_{\xi}^{p,s+\varepsilon}(K_{i})\leq C<+\infty.
$$  
$C$ being a positive constant. Then, the sequence $\bigl(E_{n}={\cup_{i,j=1}^n}(H_{i}\cap K_{j}))\bigr)_{n\in\mathbb{N}}$ is a covering of $E$. So that,
$$
\begin{array}{lll}
&&\mathcal{P}_{\xi}^{^{\alpha q+(1-\alpha)p,\alpha t+(1-\alpha)s}}(E_{n})\\
&\leq&\displaystyle{\sum_{i,j=1}^n}\mathcal{P}_{\xi}^{^{\alpha q+(1-\alpha)p,\alpha t+(1-\alpha)s}}(H_{i}\cap K_{j})\\
&\leq&\displaystyle{\sum_{i,j=1}^n}{\overline{\mathcal{P}}}_{\xi}^{^{\alpha q+(1-\alpha)p,\alpha t+(1-\alpha)s}}(H_{i}\cap K_{j}) \\
&\leq&\displaystyle\left({\sum_{i,j=1}^n}{\overline{\mathcal{P}}}_{\xi}^{q,t+\varepsilon}(H_{i}\cap K_{j})\right)^{\alpha}\left({\sum_{i,j=1}^n}{\overline{\mathcal{P}}}_{\xi}^{p,s+\varepsilon}(H_{i}\cap K_{j})\right)^{1-\alpha}  \\
&\leq&nC<\infty.
\end{array}
$$
Consequently,
$$
B_{\xi}(E_n,\alpha q+(1-\alpha)p)\leq\alpha t+(1-\alpha)s+\varepsilon,\;\forall\,\varepsilon>0.
$$
Hence,
$$
B_{\xi}(E,\alpha q+(1-\alpha)p)\leq\alpha B_{\xi}(E,q)+(1-\alpha)B_{\xi}(E,p).
$$
\textbf{d.} For $i=1,2,...,n$ and $\widehat{q_{i}}=(q_{1},...,q_{i-1},q_{i+1},...,q_{k})$ fixed and $p_{i}\leq q_{i}$ denote  
$$
q=(q_{1},...,q_{i-1},q_{i},q_{i+1},...,q_{k})\hbox{ and }
p=(q_{1},...,q_{i-1},p_{i},q_{i+1},...,q_{k}).
$$
For any $A\subseteq E$ and any centered $\varepsilon$-covering $(B(x_{i},r_{i}))_{i}$ of $A$ we have 
$$
(\mu(B(x_{i},r_{i})))^{q}(\nu(B(x_{i},r_{i})))^{t}\leq(\mu(B(x_{i},r_{i})))^{p}(\nu (B(x_{i},r_{i})))^{t},\;\forall\,t\in\mathbb{R}.
$$
Hence,
$$
{\overline{\mathcal{H}}}_{\xi,\varepsilon}^{q,t}(A)\leq{\overline{\mathcal{H}}}_{\xi,\varepsilon}^{p,t}(A),\;\forall\,A\subseteq E.
$$
When $\varepsilon\downarrow0$, we get
$$
{\overline{\mathcal{H}}}_{\xi}^{q,t}(A)\leq{\overline{\mathcal{H}}}_{\xi}^{p,t}(A),\;\forall\,A\subseteq E.
$$
Therefore,
$$
\mathcal{H}_{\xi}^{q,t}(E)\leq\mathcal{H}_{\xi}^{p,t}(E).
$$
As a result,
$$
\mathcal{H}_{\xi}^{q,t}(E)=0,\;\forall t>b_{\xi}(E,p).
$$
Consequently
$$
b_{\xi}(E,q)<t,\;\forall t>b_{\xi}(E,p).
$$
Which means that
$$
b_{\xi}(E,q)\leq b_{\xi}(E,p).
$$
The proof of the monotonicity of $B_{\xi}(E,.)$ and $\Lambda_{\xi}(E,.)$ is similar.
\end{proof}
\begin{theorem}\label{coupures-inequalities} 
Let $\mu=(\mu_{1},\mu_{2},...,\mu_{k})\in\mathcal{P}(\mathbb{R}^{d})^k$ and $\nu\in\mathcal{QAHP}(\mathbb{R}^{d})$. We have\\
\textbf{1.} $0\leq b_{\xi}(q)\leq B_{\xi}(q)\leq \Lambda_{\xi}(q)$, $\forall\,q_{i}<1$, $\forall\,1\leq i\leq k$.\\
\textbf{2.} $b_{\xi}(e_{i})=B_{\xi}(e_{i})=\Lambda_{\xi}(e_{i})=0$ with $e_{i}=(0,0,0...,1,0,0...,0).$\\
\textbf{3.} $b_{\xi}(q)\leq B_{\xi}(q)\leq \Lambda_{\xi}(q)\leq 0$,	$\forall\,q_{i}>1$, $\forall\,1\leq i\leq k$.
\end{theorem}
\begin{proof} Using (\ref{HandPborneeslunparlautre}) we get
$$
b_{\xi}(E,q)\leq B_{\xi}(E,q)\leq\Lambda_{\xi}(E,q),\;\forall q\in\mathbb{R}^{k}.
$$
We are going to prove now that $b_{\xi}(e_{i})\geq 0$ and $\Lambda
_{\xi}(e_{i})\leq 0$ with $e_{i}=(0,0,...,0,1,0,...,0)$.
Indeed if $t<0,$ $0<\varepsilon<\frac{1}{2}$ and $(B(x_{i},r_{i}))_{i}$ is an $\varepsilon$-covering of $E$, then
$$ 
{\sum_i}\mu(B(x_{i},r_{i}))^{e_{i}}\nu(B(x_{i},r_{i}))^t\geq 1\;\;\Longrightarrow\;\; {\overline{\mathcal{H}}}_{\xi,\varepsilon}^{e_{i},t}(E)\geq 1,\,\forall t>0.
$$
Therefore,
$$
t\leq b_{\xi}(e_{i}),\forall t<0.
$$
Consequently,
$$
b_{\xi}(e_{i})\geq 0.
$$
Consider now $t>0,$ $0<\delta <\frac{1}{2}$ and $(B(x_{i},r_{i}))_{i}{}$ is
a centered $\varepsilon$-packing of $E$, then
$$
{\overline{\mathcal{P}}}_{\xi,\varepsilon}^{e_{i},t}(E)\leq\underset{i}{\hbox{ }\sum}\mu(B(x_{i},r_{i}))^{e_{i}}\nu(B(x_{i},r_{i}))^t\leq 1.
$$
Consequently,
$$
{\overline{\mathcal{P}}}_{\xi,\varepsilon}^{e_{i},t}(E)\leq1,\hbox{ }\forall \hbox{ }t>0,
$$
which implies that
$$
\Lambda_{\xi}(e_{i})\leq t,\hbox{ }\forall t>0.
$$
Finally, we get
$$
\Lambda_{\xi}(e_{i})\leq 0.
$$
As a conclusion, whenever $q_{i}>1,\;\forall i=1,2,...,n$, we get $\Lambda_{\xi}(q)<\Lambda_{\xi}(e_{i})\leq 0$ and then
$$
b_{\xi}(q)\leq B_{\xi}(q)\leq\Lambda_{\xi}(q)\leq\Lambda_{\xi}(e_{i})\leq 0.
$$
Similarly, for $q_{i}<1,\;\forall i=1,2,...,n$, we obtain $b_{\xi}(q)>b_{\xi}(e_{i})\geq 0$ and thus
$$
0\leq b_{\xi}(e_{i})\leq b_{\xi}(q)\leq B_{\xi}(q)\leq\Lambda_{\xi}(q).
$$
Furthermore, we have
$$
\forall q\in\mathbb{R}^{k},\hbox{ }b_{\xi}(q)\leq B_{\xi}(q)\leq\Lambda_{\xi}(q).
$$
Then,
$$
0\leq b_{\xi}(e_{i})\leq B_{\xi}(e_{i})\leq\Lambda_{\xi}(e_{i})\leq0,
$$
which implies that
$$
b_{\xi}(e_{i})=B_{\xi}(e_{i})=\Lambda_{\xi}(e_{i})=0.
$$
\end{proof}
Next, we need to introduce the following quantities which will be useful
later. Let $\mu =(\mu_{1},\mu_{2},...,\mu_{k})$ be a vector-valued probability 
measure on $\mathbb{R}^{d}$. For $E\subseteq S_\mu$, and $a>1$ we denote
$$
T_{a}^{j}(\mu)=\displaystyle\lim\sup_{r\downarrow0}\left[\displaystyle\sup_{x\in S_{\mu}}\frac{\mu_{j}(B(x,ar))}{\mu_{j}(B(x,r))}\right],\;1\leq j\leq k,
$$
and for $x\in S_\mu$, $T_{a}^{j}(x)=T_{a}^{j}(\left\{x\right\})$.
We define the set $P_{D}(\mathbb{R}^{n})$ of doubling probability
measures on $\mathbb{R}^{n}$ by
$$
P_{D}(\mathbb{R}^{d})=\left\{\mu\in P(\mathbb{R}^{d});\hbox{ }
T_{a}^{j}(\mu)<\infty\hbox{ for some }a,\hbox{ }\forall\hbox{ }j\right\} .
$$
We denote also
$$
\mathcal{QAHP}_D(\mathbb{R}^{d})=\mathcal{QAHP}(\mathbb{R}^{d})\cap P_{D}(\mathbb{R}^{d}).
$$
Obviously, these sets are independent of $a$.
\begin{proposition}\label{pseudo-convexity-coupure-bmunu}
Let $\mu=(\mu_{1},\mu_{2},...,\mu_{k})\in\mathcal{P}(\mathbb{R}^{d})^k$ and $\nu\in\mathcal{QAHP}_D(\mathbb{R}^{d})$, $E\subset\mathbb{R}^{d}$, $p,q\in\mathbb{R}^{k}$ and $\alpha\in[0,1]$. Then, we have
$$
b_{\xi}(E,\alpha p+(1-\alpha)q)\leq\alpha B_{\xi}(E,p)+(1-\alpha)b_{\xi}(E,q).
$$
\end{proposition}
\begin{proof} Let $t=B_{\xi}(E,p)$ and $s=b_{\xi}(E,q)$. We will prove that 
$$
b_{\xi}(E,\alpha p+(1-\alpha)q)\leq\alpha t+(1-\alpha)s,\hbox{ }\forall\hbox{ }\varepsilon>0.
$$
Let $\varepsilon>0$, $m\in\mathbb{N}^{\ast}$, and denote
$$
E_{m}=\{x\in E;\;\frac{\mu_{j}(B(x,5r))}{\mu_{j}(B(x,r))}<m,\,\forall j,\,\frac{\nu(B(x,5r))}{\nu(B(x,r))}<m,\,0<r<\frac{1}{m}\}.
$$
As $E={\cup_m}E_{m}$, we shall prove that 
$$
\mathcal{H}_{\xi}^{\alpha p+(1-\alpha )q,\alpha t+(1-\alpha)s+\varepsilon}(E_{m})<\infty\hbox{ , }\forall\hbox{ }m\in\mathbb{N}^{\ast}.
$$
So, let $F\subset E_{m}$ and $(F_{i})_{i}$ be an arbitrary covering of $F$, and $\delta>0$. Let next $\varepsilon>0$, $i\in\mathbb{N}$, and $\delta_i>0$ be such that
$$
{\overline{\mathcal{P}}}_{\xi,\delta_{i}}^{p+\varepsilon,t}(F_{i})\leq{\overline{\mathcal{P}}}_{\xi}^{p+\varepsilon,t}(F_{i})+\frac{1}{2^{i}}.
$$
Since $F_{i}\cap F\subset F\subset E$, it holds that
$$
b_{\xi}(F_{i}\cap F,q)\leq b_{\xi}(E,q)=s<s+\varepsilon.
$$
Consequently
$$
b_{\xi}(F_{i}\cap F,q)<s+\varepsilon,
$$
which yields that
$$
{\overline{\mathcal{H}}}_{\xi}^{q,s+\varepsilon}(F_{i}\cap F)=0.
$$
There exists consequently a centered $(\frac{\delta}{5}\wedge\frac{1}{m}\wedge\delta_{i})$-covering $(B(x_{ij},r_{ij}))_{j\in I_{i}}$ of $F_{i}\cap F$ satisfying
$$
\displaystyle\sum_{j\in I_{i}}\mu(B(x_{ij},r_{ij}))^q\nu(B(x_{ij},r_{ij}))^{s+\varepsilon}\leq\frac{1}{2^{i}}.
$$
Let now $J_{i}\subset I_{i}$ composed of disjoint balls such that
$$
\displaystyle\bigcup_{j\in I_{i}}B(x_{ij},r_{ij})\subset\displaystyle\bigcup_{j\in J_{i}}B(x_{ij},5r_{ij}).
$$
Since $(B(x_{ij},5r_{ij}))_{j\in J_{i}}$ is a centered $\delta$-covering of $F_{i}\cap F$ and $(B(x_{ij},r_{ij}))_{j\in J_{i}}$ is a
centered $\delta_{i}$-packing of $F_{i}$, we obtain
\begin{equation}\label{eq-utile}
\begin{array}{lll}
\displaystyle{\overline{\mathcal{H}}}_{\xi,\delta}^{\alpha(p,q),\alpha_\varepsilon(t,s)}(F)
&\leq&\displaystyle{\overline{\mathcal{H}}}_{\xi,\delta}^{\alpha(p,q),\alpha_\varepsilon(t,s)}\bigl(\displaystyle\bigcup_i\displaystyle\bigcup_{j\in J_{i}}B(x_{ij},5r_{ij})\bigr)\\
&\leq&\displaystyle\sum_i\displaystyle\sum_{j\in J_{i}}\left[\mu(B(x_{ij},5r_{ij}))\right]^{\alpha(p,q)}\left[\nu(B(x_{ij},5r_{ij}))\right]^{\alpha_\varepsilon(t,s)},
\end{array}
\end{equation}
where $\alpha(p,q)=\alpha p+(1-\alpha)q$ and $\alpha_\varepsilon(t,s)=\alpha t+(1-\alpha)s+\varepsilon$. Consequently, whenever $\alpha(p,q)\in(0,+\infty)^k$ and $\alpha_\varepsilon(t,s)\in(0,+\infty)$, we get 
$$
\left[\mu(B(x_{ij},5r_{ij}))\right]^{\alpha(p,q)}\leq\,m^{|\alpha(p,q)|}\bigl[\mu(B(x_{ij},r_{ij}))\bigr]^{\alpha(p,q)},
$$
and
$$
\left[\nu(B(x_{ij},5r_{ij}))\right]^{\alpha_\varepsilon(t,s)}\leq\,m^{\alpha_\varepsilon(t,s)}\bigl[\nu(B(x_{ij},r_{ij}))\bigr]^{\alpha_\varepsilon(t,s)}.
$$
Consequently, using (\ref{eq-utile}), we get
$$
{\overline{\mathcal{H}}}_{\xi,\delta}^{\alpha(p,q),\alpha_\varepsilon(t,s)}(F)\leq m^{|\alpha(p,q)|+\alpha_\varepsilon(t,s)}({\sum_i}({\overline{\mathcal{P}}}_{\xi}^{p,t+\varepsilon}(F_{i})+\frac{1}{2^{i}}))^{\alpha},
$$
which yields that
$$
{\overline{\mathcal{H}}}_{\xi}^{\alpha(p,q),\alpha_\varepsilon(t,s)}(F)\leq m^{|\alpha(p,q)|+\alpha_\varepsilon(t,s)}({\sum_i}{\overline{\mathcal{P}}}_{\xi}^{p,t+\varepsilon }(F_{i})+1)^{\alpha}.
$$
Hence, $\forall \hbox{ }F\subseteq E_{m}$, we get
$$
{\overline{\mathcal{H}}}_{\xi}^{\alpha(p,q),\alpha_\varepsilon(t,s)}(F)\leq m^{|\alpha(p,q)|+\alpha_\varepsilon(t,s)}({{\mathcal{P}}}_{\xi}^{p,t+\varepsilon}(F)+1)^{\alpha},
$$
which implies that 
$$
{\overline{\mathcal{H}}}_{\xi}^{\alpha(p,q),\alpha_\varepsilon(t,s)}(E_m)\leq m^{|\alpha(p,q)|+\alpha_\varepsilon(t,s)}({{\mathcal{P}}}_{\xi}^{p,t+\varepsilon}(E_m)+1)^{\alpha}.
$$
Consequently, 
$$
{\overline{\mathcal{H}}}_{\xi}^{\alpha(p,q),\alpha_\varepsilon(t,s)}(E_m)<\infty,\forall m.
$$
Therefore,
$$
b_{\xi}(E_m,\alpha(p,q))\leq\alpha_\varepsilon(t,s),\forall\varepsilon>0,\,\forall m,
$$
which yields finally that
$$
b_{\xi}(E,\alpha p+(1-\alpha)q)\leq\alpha t+(1-\alpha)s=\alpha B_{\xi}(E,p)+(1-\alpha)b_{\xi}(E,q).
$$
\end{proof}
\begin{theorem}\label{corollaire-1-coupures}
Let $\mu=(\mu_{1},\mu_{2},...,\mu_{k})\in\mathcal{P}(\mathbb{R}^{d})^k$ and $\nu\in\mathcal{QAHP}_D(\mathbb{R}^{d})$, $q\in\mathbb{R}^{k}$ and $E\subseteq S_{\mu}\cap S_{\nu}$. The following assertions hold.\\
1. Whenever $q_{i}\leq 0,\,\forall k$, we have
\begin{equation}\label{eq1-proof-cor-1coupures}
b_{\xi}^{q}(E)\geq\dim_{\nu}(E)\left(1-\frac{|q|}{k}\right).
\end{equation}
2. Whenever $0\leq q_{i}\leq 1,\,\forall k$, we have
\begin{equation}\label{eq2-proof-cor-1coupures}
b_{\xi}^{q}(E)\leq\dim_{\nu}(E)\left(1-\frac{|q|}{k}\right)\leq\frac{n}{k}(k-|q|).
\end{equation}
3. Whenever $q_{i}\geq 1,\,\forall k$, we have
\begin{equation}\label{eq3-proof-cor-1coupures}
b_{\xi}^{q}(E)\geq\frac{\beta}{\beta-1}\dim_{\nu}(E)\geq \frac{\beta}{\beta-1}n,\mbox{ \ with }\beta ={\max_i}(1-\frac{1}{q_{i}}).
\end{equation}
\end{theorem}
\begin{proof} \textbf{1.} For $q=(q_{1},...,q_{k})\in\mathbb{R}^{k}$ take in Proposition \ref{pseudo-convexity-coupure-bmunu}, $p=e_i$, $\widetilde{q}_i=q_ie_i$ and $\alpha=\frac{-q_{i}}{1-q_{i}}$. As $q_{i}\leq0,\,\forall i$, we get in one hand
$$
b_{\xi}(0)\leq\alpha B_{\xi}(e_i)+(1+\frac{q_{i}}{1-q_{i}})b_{\xi}(\widetilde{q}_i).
$$
Recall now that $B_{\xi}(e_i)=0$. Therefore,
$$
(1-q_{i})b_{\xi}(0)\leq b_{\xi}(\widetilde{q}_i)\leq b_{\xi}(q),
$$
which implies that
$$
(1-q_{i})\dim_{\nu}(E)\leq b_{\xi}^{q}(E).
$$
The summation on $i=1,2,...,k$ gives
$$
b_{\xi}^{q}(E)\geq\dim_{\nu}(E)\bigl(1-\frac{|q|}{k}\bigr).
$$
\textbf{2.} For $q=(q_{1},...,q_{k})\in\mathbb{R}^{k}$ take in Proposition \ref{pseudo-convexity-coupure-bmunu}, $p=e_i$, $\widetilde{q}_i=q_ie_i$, and $\alpha=q_{i}$, and follow similar techniques as in assertion \textbf{1.}\\
\textbf{3.} For $q=(q_{1},...,q_{k})\in\mathbb{R}^{k}$ take in Proposition \ref{pseudo-convexity-coupure-bmunu}, $p=0$, and $\alpha=\beta$, and follow as usual similar techniques as previously.
\end{proof}
\begin{theorem}\label{corollaire-2-coupures}
Let $\xi=(\mu,\nu)=(\mu_{1},\mu_{2},...,\mu_{k},\nu)\in\mathcal{P}(\mathbb{R}^{d})^k\times\in\mathcal{QAHP}_D(\mathbb{R}^{d})$, $q\in\mathbb{R}^{k}$, and $E\subseteq S_{\mu}\cap S_{\nu}$. The following assertions hold.\\
1. Whenever $q_{i}\leq 0,\,\forall k$, we have
\begin{equation}\label{eq1-proof-cor-2coupures}
B_{\xi}^{q}(E)\geq\,Dim_{\nu}(E)\left(1-\frac{|q|}{k}\right).
\end{equation}
2. Whenever $0\leq q_{i}\leq 1,\,\forall k$, we have
\begin{equation}\label{eq2-proof-cor-2coupures}
B_{\xi}^{q}(E)\leq\,Dim_{\nu}(E)\left(1-\frac{|q|}{k}\right)\leq\frac{n}{k}(k-|q|).
\end{equation}
3. Whenever $q_{i}\geq 1,\,\forall k$, we have
\begin{equation}\label{eq3-proof-cor-2coupures}
B_{\xi}^{q}(E)\geq\frac{\beta}{\beta-1}\,Dim_{\nu}(E)\geq \frac{\beta}{\beta-1}n,\mbox{ \ with }\beta =\hbox{ }\underset{i}{\max}(1-\frac{1}{q_{i}}).
\end{equation}
\end{theorem}
The proof follows similar techniques as in Theorem \ref{corollaire-1-coupures}.
\begin{proposition}\label{lowerboundbmunu} Let $\xi=(\mu,\nu)=(\mu_{1},\mu_{2},...,\mu_{k},\nu)\in\mathcal{P}(\mathbb{R}^{d})^k\times\mathcal{AHP}_D(\mathbb{R}^{d})$ be Radon vector-valued measure on $\mathbb{R}^{d}$ with compact support, and such that $S_{\mu}\subseteq S_{\nu}$. Assume further that $\mu$ is absolutely continuous relatively to the Lebesgue measure on $S_{\mu_{i}}$. Then, for all Borel set $E\subset S_{\mu}$ such that $\mu_i(E)>0$, $\forall i$ and $q\in\mathbb{R}^k$ with $|q|\in[0,1]$, we have
$$
\alpha b_{\xi}(E,q)\geq(1-|q|),
$$
with $|q|=q_{1}+q_{2}+...+q_{k}$ and $\alpha$ is the Ahlfors regularity index of $\nu$.
\end{proposition}
\begin{proof} Since $\mu$ is absolutely continuous relatively to the Lebesgue measure on $S_{\mu_{i}}$, we may find for each $i$ a function $g_{i}\geq0$ such that $\mu_{i}=g_{i}\lambda_{|_{S_{\mu_{i}}}}^{n}$. Therefore, as $\mu_i(E)>0$, there exists a $B_{i}\subset E$ (being Borel) with Lebesgue measure $\lambda^n(B_i)>0$, and a constant $\gamma_{i}>0$ satisfying   
\begin{equation}\label{eq-utile-2}
g_{i}(x)\geq\gamma_{i},\,\forall x\in B_{i}.
\end{equation}
Let next $\varepsilon>0$ and $(B(x_{i},r_{i}))_{i}$ be a centred $\varepsilon$-covering of $E$. For $t>0$, we have
$$
\underset{i}{\sum}(\mu(B(x_{i},r_{i})))^{q}(\nu(B(x_{i},r_{i})))^{t}
\geq C_\nu\underset{i}{\sum}(\mu(B(x_{i},r_{i})))^{q}(\lambda
^{n}(B(x_{i},r_{i})))^{\alpha t},
$$
where $C_\nu$ is a constant due to Ahlfors regularity of $\nu$. Denote next $B=\underset{i}{\cup }B_{i}\subset E$. It holds from (\ref{eq-utile-2}) that
$$
(\mu(B(x_{i},r_{i})))^{q}\geq\gamma^{q}\Bigl(\lambda^{n}(B(x_{i},r_{i})\cap B\Bigr)^{q}.
$$
As a result, 
$$
\underset{i}{\sum}(\mu(B(x_{i},r_{i})))^{q}(\nu(B(x_{i},r_{i})))^{t}
\geq\,C_\nu\gamma^{q}\underset{i}{\sum}\Bigl(\lambda^{n}(B(x_{i},r_{i})\cap B)\Bigr)^{|q|+\alpha t}.
$$
For $\alpha t<1-|q|$, it yields that
$$
\underset{i}{\sum}(\mu(B(x_{i},r_{i})))^{q}(\nu(B(x_{i},r_{i})))^{t}
\geq\,C_\nu\gamma^{q}\lambda^{n}(B)>0.
$$
Consequently, $\forall\,t$ such that $0<\alpha t<1-|q|$, we get 
$$
\mathcal{H}_{\xi}^{q,t}(E)>0,
$$
which implies that
$$
b_{\xi}(E,q)\geq t.
$$
By letting $t\rightarrow\displaystyle\frac{1-|q|}{\alpha}$, we obtain
$$
b_{\xi}^{q}(E)\geq\displaystyle\frac{1-|q|}{\alpha}.
$$
\end{proof}
\begin{proposition}\label{lowerboundBmunu} Let $p>1$, $\xi=(\mu,\nu)=(\mu_{1},\mu_{2},...,\mu_{k},\nu)\in\mathcal{P}(\mathbb{R}^d)^k\times\mathcal{AHP}(\mathbb{R}^d)$ be a vector-valued Radon probability measure on $\mathbb{R}^{d}$ with compact support, and $S_{\mu}\subset S_{\nu}$. Suppose further that $\mu_{i}\in L^{p}(\mathbb{R}^d)$. Then for $q_{i}\geq 1,$ we have
$$
\alpha B_{\xi}(q)\leq\max\left\{k-|q|,\frac{-|q|(p-1)}{p}\right\}.
$$
\end{proposition}
\begin{proof} Let for $i=1,2,...,k$, $g_i\in L^p(\mathbb{R}^d)$ be such $d\mu_i=g_id\lambda^n$ on $S_{\mu_{i}}$. Of course, the $g_{i}$'s are compactly supported functions. Assume for instance that $q_{i}\geq p>1$, $\forall i$. Let next $\delta >0$, and consider a centred $\delta$-packing $(B(x_{i},r_{i}))_{i}$ of $S_{\mu }$. Let finally, $g=\underset{i}{\max}g_{i}$. Analogously to Proposition \ref{lowerboundbmunu}, it follows for $t>\frac{-|q|(p-1)}{p}$, that
$$
\underset{i}{\sum}(\mu(B(x_{i},r_{i})))^{q}(\nu(B(x_{i},r_{i})))^{t}
\leq C_\nu\underset{i}{\sum}\left(\lambda^{n}(B(x_{i},r_{i}))\right)^{\alpha t+\frac{|q|(p-1)}{p}}\left(\underset{B(x_{i},r_{i})}{\int g^{p}}d\lambda^{n}\right)^{{\frac{|q|}{p}}}.
$$
As $\alpha t>\frac{-|q|(p-1)}{p}$, we get
$$
\underset{i}{\sum}(\mu(B(x_{i},r_{i})))^{q}(\nu(B(x_{i},r_{i})))^{t}
\leq C_\mu\underset{i}{\sum}\left(\underset{B(x_{i},r_{i})}{\int g^{p}}d\lambda^{n}\right).
$$
This yields that
$$
\underset{i}{\sum}(\mu(B(x_{i},r_{i})))^{q}(\nu(B(x_{i},r_{i})))^{t}
\leq C_\mu(\int g^{p}d\lambda^{n})^{\frac{|q|}{p}}<\infty.
$$
Hence,
$$
{\overline{\mathcal{P}}}_{\xi}^{q,t}(S_{\mu })<\infty,
$$
and consequently,
$$
\mathcal{P}_{\xi}^{q,t}(S_{\mu })<\infty.
$$
Therefore,
$$
B_{\xi}(E,q)\leq t,\;\forall t>\frac{-|q|(p-1)}{\alpha p}.
$$
As a result
$$
\alpha B_{\xi}(E,q)\leq\frac{-|q|(p-1)}{p}.
$$
Now, assume that $1\leq q_{i}<p$, $\forall i$. Let $t>n-|q|$, $\delta>0$, and consider a centred $\delta$-packing $(B_i=B(x_{i},r_{i}))_{i}$ of $S_{\mu}$. We get
$$
\underset{i}{\sum}(\mu(B_{i}))^{q}(\nu(B_{i}))^{t}\leq C_\mu\underset{i}{\sum}(\mu(B_{i}))^{q}((\lambda ^{n}(B_{i}\cap S_{\mu}))^{\alpha t}).
$$
Therefore,
$$
\underset{i}{\sum}(\mu(B_{i}))^{q}(\nu(B_{i}))^{t}\leq C_\mu
\underset{i}{\sum}\prod_l\left(\underset{B_{i}\cap S_{\mu}}{\int
g_{l}}d\lambda^{n}\right)^{{q_{l}}}(\lambda^{n}(B_{i}\cap S_{\mu}))^{\alpha t}.
$$
Next, by H\"older's inequality, it follows that
$$
\underset{i}{\sum}(\mu(B_{i}))^{q}(\nu(B_{i}))^{t}\leq C_\mu
\underset{i}{\sum}\prod_l\left(\underset{B_{i}\cap S_{\mu}}{\int
g_{l}^{q_{l}}}d\lambda^{n}\right)\left[\lambda^{n}(B_{i}\cap S_{\mu})\right]^{q_{l}-1}\lambda^{n}(B_{i}\cap S_{\mu}))^{\alpha t}.
$$
This yields that
$$
\underset{i}{\sum}(\mu(B_{i}))^{q}(\nu(B_{i}))^{t}\leq C_\mu
\underset{i}{\sum}(\lambda^{n}(B_{i}\cap S_{\mu}))^{\alpha t+|q|-k}\prod_l\left(\underset{B_{i}\cap S_{\mu}}{\int	g^{q_{l}}}d\lambda^{n}\right)<C<\infty.
$$
As a consequence, we get
$$
{\overline{\mathcal{P}}}_{\xi}^{q,t}(S_{\mu })<\infty,
$$
which means that
$$
B_{\xi}(q)\leq t,\;\forall t>\displaystyle\frac{k-|q|}{\alpha}.
$$
This in turns yields that
$$
B_{\xi}(q)\leq\displaystyle\frac{k-|q|}{\alpha}.
$$
\end{proof}
\begin{remark}
For $\alpha=1$, the measure $\nu$ is equivalent to the Lebesgue's one. If further $k=1$, Propositions \ref{lowerboundbmunu} and \ref{lowerboundBmunu} are the classical cases raised by Olsen.
\end{remark}
\section{The associated joint multifractal spectrum}
We propose in this section to introduce an associated multifractal spectrum relatively to the joint generalisations of Hausdorff measure/dimension and their analogues of packing measure/dimension developed in the previous section. 

Consider a vector $\xi=(\mu,\nu)=(\mu_{1},\mu_{2},...,\mu_{k},\nu)$ of Borel probability measures on $\mathbb{R}^{d}$. We define the local upper dimension of $\mu_{j}$ relatively to $\nu$ at $x\in\mathbb{R}^{d}$ by
$$
{\overline{\mathcal{\gamma }}}_{\mu_{j},\nu}(x)=\underset{r\rightarrow 0}{\hbox{lim sup}}\frac{\log (\mu_{j}(B(x,r)))}{\log (\nu (B(x,r)))}.
$$
Similarly, the local lower dimension of $\mu_{j}$ relatively to $\nu$ at $x\in\mathbb{R}^{d}$ is 
$$
{\underline{\gamma }}_{\mu_{j,\nu}}(x)=\underset{r\rightarrow 0}{\hbox{	lim inf}}\frac{\log (\mu_{j}(B(x,r)))}{\log (\nu (B(x,r)))}.
$$
Whenever ${\overline{\mathcal{\gamma }}}_{\mu_{j},\nu}(x)={\underline{\gamma }}_{\mu_{j,\nu}}(x)$, we call local dimension of $\mu_{j}$ relatively to $\nu $ at $x$ their common value, which will be written ${\gamma}_{\mu_{j},\nu}(x)$. Next, for $\gamma=(\gamma_{1},\gamma_{2},\dots,\gamma_{k})\in\mathbb{R}_{+}^{k}$, we set
$$
{\overline{\mathcal{K}}}^{\gamma}=\left\{x\in{S}_\xi\,;\;|{\overline{\mathcal{\gamma}}}_{\mu_{j},\nu}(x)\leq\gamma_{j}\,\;\forall j=1,\dots,k\right\},
$$
$$
{\overline{\mathcal{K}}}_{\gamma}=\left\{x\in{S}_\xi\,;\;\gamma_{j}\leq{\overline{\mathcal{\gamma}}}_{\mu_{j},\nu}(x),\;\forall j=1,\dots,k\right\},
$$
$$
{\underline{K}}^{\gamma}=\left\{x\in{S}_\xi\,;\;{\underline{\gamma}}_{\mu_{j},\nu}(x)\leq\gamma_{j},\;\forall j=1,\dots,k\right\}
$$
and
$$
{\underline{K}}_{\gamma}=\left\{x\in{S}_\xi\,;\;\gamma_j\leq{\underline{\gamma}}_{\mu_{j},\nu}(x),\;\forall j=1,\dots,k\right\}.
$$
Let also 
$$
K(\gamma)={\underline{K}}_{\gamma}\cap{\overline{\mathcal{K}}}^{\gamma}.
$$
The joint spectrum of singularities associated to the vector $\xi$ is
$$
d(\gamma)=dim K(\gamma),
$$
where the symbol dim designates the Hausdorff dimension. In the present section, we aim to establish some bounds for such a spectrum.
\begin{proposition} Consider a metric space $X$, and a vector-valued Borel probability measure $\xi=(\mu,\nu)=(\mu_{1},\mu_{2},...,\mu_{k},\nu)$ on $\mathbb{R}^{d}$. Fix $\gamma =(\gamma _{1},\gamma _{2},...,\gamma _{k})\in\mathbb{R}_{+}^{k}$, $q=(q_{1},q_{2},...,q_{k})\in\mathbb{R}^{k}$, $t\in\mathbb{R}$, and $\delta>0$ satisfying $0<\delta\leq\langle\gamma,q\rangle+t$. The following assertions are true.
\begin{enumerate}
\item $\mathcal{H}_{\nu}^{\langle\gamma,q\rangle+t+k\delta}({\overline{\mathcal{K}}}^{\gamma})\leq\mathcal{H}_{\xi}^{q,t}({\overline{\mathcal{K}}}^{\gamma})$ for $q\in\mathbb{R}_{+}^{k}.$
\item $\mathcal{H}_{\nu}^{\langle\gamma,q\rangle+t+k\delta}({\underline{K}}_{\gamma })\leq\mathcal{H}_{\xi}^{q,t}({\underline{K}}_{\gamma})$ for $q\in\mathbb{R}_{-}^{k}.$
\item $\mathcal{P}_{\nu}^{\langle\gamma,q\rangle+t+k\delta}({\overline{\mathcal{K}}}^{\gamma})\leq\mathcal{P}_{\xi}^{q,t}({\overline{\mathcal{K}}}^{\gamma})$ for $q\in\mathbb{R}_{+}^{k}.$
\item $\mathcal{P}_{\nu}^{\langle\gamma,q\rangle+t+k\delta}({\underline{K}}_{\gamma })\leq\mathcal{P}_{\xi}^{q,t}({\underline{K}}_{\gamma})$ for $q\in\mathbb{R}_{-}^{k}.$
\item For all Borel set $A\subseteq{\overline{\mathcal{K}}}^{\gamma}$,  $\mathcal{H}_{\xi}^{q,t}({A})\leq\mathcal{H}_{\nu}^{\langle\gamma,q\rangle+t-k\delta}({\overline{\mathcal{K}}}^{\gamma})$ if $q\in\mathbb{R}_{-}^{k}.$
\item For all Borel set $A\subseteq{\underline{K}}_{\gamma}$, $\mathcal{H}_{\xi}^{q,t}({A})\leq\mathcal{H}_{\nu}^{\langle\gamma,q\rangle+t-k\delta}({\overline{\mathcal{K}}}^{\gamma})$ if $q\in\mathbb{R}_{+}^{k}.$
\item For all Borel set $A\subseteq{\overline{\mathcal{K}}}^{\gamma}$, $\mathcal{P}_{\xi}^{q,t}({A})\leq\mathcal{P}_{\nu}^{\langle\gamma,q\rangle+t-k\delta}({\overline{\mathcal{K}}}^{\gamma})$ if $q\in\mathbb{R}_{-}^{k}.$
\item For all Borel set $A\subseteq{\underline{K}}_{\gamma}$, $\mathcal{P}_{\xi}^{q,t}({A})\leq\mathcal{P}_{\nu}^{\langle\gamma,q\rangle+t-k\delta}({\overline{\mathcal{K}}}^{\gamma})$ if $q\in\mathbb{R}_{+}^{k}.$
\end{enumerate}
\end{proposition}
\begin{proof} We will develop the proofs of assertions 1., 3., 5. and 7. The remaining assertions may be shown by similar arguments.\\
1. For $q=0$ the statement is obvious. So, let next $q_j>0$, for all $j$.  For $m\in\mathbb{N}$, let
$$
{\overline{\mathcal{K}}}_{m}^{\gamma}=\left\{x\in{\overline{\mathcal{K}}}^{\gamma};\;\frac{\log(\mu_{j}(B(x,r)))}{\log(\nu(B(x,r)))}\leq\gamma_{j}\hbox{}+\frac{\delta}{q_{j}},\;0<r<\frac{1}{m},\,\forall j,\hbox{ s.t. }1\leq j\leq k\right\}.
$$
Let $\eta$ be such that $0<\eta<\frac{1}{m}$ and $(B_{i}=B(x_{i},r_{i}))_{i}$
be an $\eta$-covering of $E\subseteq{\overline{\mathcal{K}}}_{m}^{\gamma}$. It follows that
$$
\frac{\log(\mu_{j}(B(x,r)))}{\log(\nu(B(x,r)))}\leq\gamma_{j}\hbox{}+\frac{\delta}{q_{j}}\Rightarrow\mu_{j}(B(x,r))\geq\nu(B(x,r)^{\gamma_{j}\hbox{}+\frac{\delta}{q_{j}}}.
$$
Since $q\in\mathbb{R}_{+}^{k}$, we get
$$
\mu(B(x,r))^q\geq\nu(B(x,r)^{\langle\gamma,q\rangle+k\delta}.
$$
Then,
$$
\mathcal{H}_{\nu,\eta}^{\langle\gamma,q\rangle+t+k\delta}({E})
\leq\underset{i}{\hbox{}\sum}\nu(B(x,r)^{\langle\gamma,q\rangle+t+k\delta}\leq\underset{i}{\hbox{}\sum}\mu(B(x,r))^q\nu(B(x,r)^{t},
$$
which implies that
$$
\mathcal{H}_{\nu,\eta}^{\langle\gamma,q\rangle+t+k\delta}({E})\leq\mathcal{H}_{\xi,\eta}^{q.t}({E}),\forall\hbox{}\eta>0.
$$
When $\eta\downarrow0$, we get for all $E\subseteq{\overline{\mathcal{K}}}_{m}^{\gamma}$,
$$
\mathcal{H}_{\nu}^{\langle\gamma,q\rangle+t+k\delta}({E})\leq\mathcal{H}_{\xi}^{q.t}({E})\leq\mathcal{H}_{\xi}^{q.t}({\overline{\mathcal{K}}}_{m}^{\gamma}).
$$
Hence, 
$$
\mathcal{H}_{\nu}^{\langle\gamma,q\rangle+t+k\delta}({\overline{\mathcal{K}}}_{m}^{\gamma})\leq\mathcal{H}_{\xi}^{q.t}({\overline{\mathcal{K}}}_{m}^{\gamma}).
$$
Finally, the result follows since ${\overline{\mathcal{K}}}^{\gamma}=\underset{m}{\cup}{\overline{\mathcal{K}}}_{m}^{\gamma}.$\\
3. For $m\in\mathbb{N}$, denote
$$
{\overline{\mathcal{K}}}_{m}^{\gamma}=\left\{x\in{\overline{\mathcal{K}}}^{\gamma};\;\frac{\log(\mu_{j}(B(x,r)))}{\log(\nu(B(x,r)))}\leq\gamma_{j}+\frac{\delta}{q_{j}},\,0<r<\frac{1}{m},\forall\,j,\hbox{ s.t. }1\leq j\leq k\right\}.
$$
Let $\eta$ be such that $0<\eta<\frac{1}{m}$ and $(B_{i}=B(x_{i},r_{i}))_{i}$ be an $\eta$-packing of $E\subseteq{\overline{\mathcal{K}}}_{m}^{\gamma}$. We have
$$
\frac{\log(\mu_{j}(B(x,r)))}{\log(\nu(B(x,r)))}
\leq\gamma_{j}\hbox{}+\frac{\delta}{q_{j}}\Rightarrow\mu_{j}(B(x,r))\geq\nu(B(x,r))^{\gamma_{j}\hbox{}+\frac{\delta}{q_{j}}}.
$$
Since $q\in\mathbb{R}_{-}^{k},$ then 
$$
\mu(B(x,r))^q\geq\nu(B(x,r))^{\langle\gamma,q\rangle+k\delta},
$$
which implies that
$$
\underset{i}{\hbox{}\sum}\nu(B(x,r))^{\langle\gamma,q\rangle+t+k\delta}
\leq\underset{i}{\hbox{}\sum}\mu(B(x,r))^q\nu(B(x,r))^{t}\leq{\overline{\mathcal{P}}}_{\xi,\eta}^{q,t}(E).
$$
Therefore,
$$
{\overline{\mathcal{P}}}_{\nu,\eta}^{\langle\gamma,q\rangle+t+k\delta}(E)\leq{\overline{\mathcal{P}}}_{\xi,\eta}^{q,t}(E),\;\forall\hbox{}\eta>0.
$$
By letting $\eta\downarrow 0$ we obtain
$$
{\overline{\mathcal{P}}}_{\nu}^{\langle\gamma,q\rangle+t+k\delta}({E})\leq{\overline{\mathcal{P}}}_{\xi}^{q.t}({E}),\;\forall\hbox{}\eta>0\hbox{ and }\forall \hbox{}E\subseteq{\overline{\mathcal{K}}}_{m}^{\gamma}.
$$
Consequently, for any covering $(E_{i})_{i}$ of ${\overline{\mathcal{K}}}_{m}^{\gamma}$, we get
$$
\begin{array}{lll}
\mathcal{P}_{\nu}^{\langle\gamma,q\rangle+t+k\delta}({\overline{\mathcal{K}}}_{m}^{\gamma})&=&\mathcal{P}_{\nu}^{\langle\gamma,q\rangle+t+k\delta}(\underset{i}{\cup}({\overline{\mathcal{K}}}_{m}^{\gamma}\cap E_{i}))\\
&\leq&\displaystyle\sum_i\mathcal{P}_{\nu}^{\langle\gamma,q\rangle+t+k\delta}({\overline{\mathcal{K}}}_{m}^{\gamma}\cap E_{i})\\
&\leq&\displaystyle\sum_i{\overline{\mathcal{P}}}_{\nu}^{\langle\gamma,q\rangle+t+k\delta}({\overline{\mathcal{K}}}_{m}^{\gamma}\cap E_{i})\\
&\leq&\displaystyle\sum_i{\overline{\mathcal{P}}}_{\xi}^{q.t}({E}_{i}),
\end{array}
$$
which leads to 
$$
\mathcal{P}_{\nu}^{\langle\gamma,q\rangle+t+k\delta}({\overline{\mathcal{K}}}_{m}^{\gamma})\leq\mathcal{P}_{\xi}^{q,t}({\overline{\mathcal{K}}}_{m}^{\gamma}),\;\forall\hbox{}m>0.
$$
Finally, since ${\overline{\mathcal{K}}}^{\gamma}=\underset{m}{\cup}{\overline{\mathcal{K}}}_{m}^{\gamma}$, we obtain
$$
\mathcal{P}_{\nu}^{\langle\gamma,q\rangle+t+k\delta}({\overline{\mathcal{K}}}^{\gamma})\leq\mathcal{P}_{\xi}^{q,t}({\overline{\mathcal{K}}}^{\gamma}).
$$
5. For $m\in\mathbb{N}$, consider
$$
{T}_{m}=\left\{x\in{A};\;\frac{\log(\mu_{j}(B(x,r)))}{\log(\nu(B(x,r)))}\leq\gamma_{j}-\frac{\delta}{q_{j}},\,0<r<\frac{1}{m},\;\forall\,j,\hbox{ s.t. }1\leq j\leq k\right\}.
$$
Consider a subset $E\subseteq{T}_{m}$, $0<\eta<\frac{1}{m}$, and a covering $(E_i)_{i\in\mathbb{N}}$ of $E$ with $r_{i}=diam E_{i}<\eta$, $\forall i$. Consider next the set $I=\left\{i;\;E_{i}\cap E=\emptyset\right\}$, and let $x_{i}\in E_{i}\cap E$. We get an $\eta$-covering $(B(x_{i},r_{i}))_{i}$ of $E$ for which we may write that
$$
\frac{\log(\mu_{j}(B(x_{i},r_{i})))}{\log(\nu(B(x_{i},r_{i})))}\leq\gamma_{j}\hbox{}-\frac{\delta}{q_{j}}.
$$
Thus,
$$
\mu(B(x_{i},r_{i}))^q\geq\nu(B(x_{i},r_{i}))^{\langle\gamma,q\rangle-k\delta}.
$$
Therefore, 
$$
\begin{array}{lll}
{\overline{\mathcal{H}}}_{\xi,\eta}^{q.t}({E}) 
&\leq&\displaystyle\sum_{i\in{I}}\mu(B(x_{i},r_{i}))^q\nu(B(x_{i},r_{i}))^{t}\\
&\leq&\displaystyle\sum_{i\in{I}}\left[\nu(B(x_{i},r_{i}))\right]^{\langle\gamma,q\rangle-k\delta+t}\\
&\leq&\displaystyle\sum_{i\in{I}}\left[\nu(B(x_{i},r_{i}))\right]^{\langle\gamma,q\rangle-k\delta+t}.
\end{array}
$$
Consequently, 
$$
{\overline{\mathcal{H}}}_{\xi,\eta}^{q.t}({E})\leq\mathcal{H}_{\nu,\eta}^{\langle\gamma,q\rangle-k\delta+t}({E})\hbox{ for }\eta<\frac{1}{m}.
$$
By letting $\eta\rightarrow0$, we obtain
$$
\overline{\mathcal{H}}_{\xi}^{q,t}(E)\leq\mathcal{H}_{\nu}^{\langle\gamma,q\rangle-k\delta+t}(E)\leq\mathcal{H}_{\nu}^{\langle\gamma,q\rangle-k\delta+t}(T_m),\;\forall E\subseteq T_{m}.
$$
Then,
$$
\mathcal{H}_{\xi}^{q,t}({T}_{m})\leq\mathcal{H}_{\nu}^{\langle\gamma,q\rangle-k\delta+t}({T}_{m})\hbox{},\;\forall\hbox{}m\in\mathbb{N}.
$$
Finally, the result follows from the equality ${A}=\underset{m}{\cup}{T}_{m}.$\\
7. Denote for $m\in\mathbb{N}$,
$$
{T}_{m}=\left\{x\in{A};\;\frac{\log(\mu_{j}(B(x,r)))}{\log(\nu(B(x,r)))}\leq\gamma_{j}-\frac{\delta}{q_{j}},\,0<r<\frac{1}{m},\,\forall j\hbox{ s.t. }1\leq j\leq k\right\}.
$$
Given $E\subseteq{T}_{m}$, $0<\eta<\frac{1}{m}$ and an $\eta$-packing $(B(x_{i},r_{i}))_{i}$ of $E$, we get 
$$
\frac{\log(\mu_{j}(B(x_{i},r_{i})))}{\log(\nu(B(x_{i},r_{i})))}\leq\gamma_{j}\hbox{}-\frac{\delta}{q_{j}}\Rightarrow\mu(B(x_{i},r_{i}))^q\leq\nu(B(x_{i},r_{i}))^{\langle\gamma,q\rangle-k\delta}.
$$
Then
$$
\displaystyle\sum_{i}\mu(B(x_{i},r_{i}))^q\nu(B(x_{i},r_{i}))^{t}\leq\displaystyle\sum_{i}\left[\nu(B(x_{i},r_{i}))\right]^{\langle\gamma,q\rangle-k\delta+t}\leq{\overline{\mathcal{P}}}_{\nu,\eta}^{\langle\gamma,q\rangle-k\delta+t}(E).
$$
Consequently, 
$$
{\overline{\mathcal{P}}}_{\xi,\eta}^{q,t}(E)\leq{\overline{\mathcal{P}}}_{\nu,\eta}^{\langle\gamma,q\rangle-k\delta+t}(E).
$$
By letting $\eta\rightarrow0$ we obtain
$$
{\overline{\mathcal{P}}}_{\xi}^{q,t}(E)\leq{\overline{\mathcal{P}}}_{\nu}^{\langle\gamma,q\rangle-k\delta+t}(E),\hbox{}\forall\hbox{}E\subset E_{m}.
$$
Now, as $(E_{i})_{i}$ is a covering of ${T}_{m}$, we get
$$
\begin{array}{lll}
\mathcal{P}_{\xi}^{q,t}({T}_{m})&\leq&\mathcal{P}_{\xi}^{q,t}(\underset{i}{\cup}({T}_{m}\cap E_{i}))\\
&\leq&\displaystyle\sum_{i}\mathcal{P}_{\xi}^{q,t}({T}_{m}\cap E_{i})\\
&\leq&\displaystyle\sum_{i}{\overline{\mathcal{P}}}_{\xi}^{q,t}({T}_{m}\cap E_{i})\\
&\leq&\displaystyle\sum_{i}{\overline{\mathcal{P}}}_{\nu}^{\langle\gamma,q\rangle-k\delta+t}({T}_{m}\cap E_{i}).
\end{array}
$$
Finally, as $A=\underset{m}{\cup}{T}_{m}$, we obtain 
$$
\mathcal{P}_{\xi}^{q,t}({A})\leq\mathcal{P}_{\nu}^{\langle\gamma,q\rangle+t-k\delta}({\overline{\mathcal{K}}}^{\gamma}),\;\hbox{ for }q\in\mathbb{R}_{-}^{k}.
$$
\end{proof}
\begin{corollary}\label{corollaire-1-spectre} Consider a metric space $X$, and a vector-valued Borel probability measure $\xi=(\mu,\nu)=(\mu_{1},\mu_{2},...,\mu_{k},\nu)$ on $\mathbb{R}^{d}$. Fix $\gamma=(\gamma_{1},\gamma_{2},...,\gamma_{k})\in\mathbb{R}_{+}^{k}$, and $q=(q_{1},q_{2},...,q_{k})\in\mathbb{R}^{k}$. Then, the following assertions hold.\\
1. Whenever $\langle\gamma,q\rangle+b_{\xi}(q)\geq0$, we have\\
$$
\left\{\begin{array}{lll}
\dim_{\nu}({\overline{\mathcal{K}}}^{\gamma})\leq\langle\gamma,q\rangle+b_{\xi}(q)\;\hbox{for}\;q\in\mathbb{R}_{+}^{k},\\
\\
\dim_{\nu}({\underline{K}}_{\gamma})\leq\langle\gamma,q\rangle+b_{\xi}(q)\;\hbox{for}\;q\in\mathbb{R}_{-}^{k}.
\end{array}
\right.	
$$
In particular, $\dim_{\nu}({\overline{\mathcal{K}}}^{\gamma})\leq\gamma_{i}$, $\forall1\leq i\leq k.$\\
2. Whenever $\langle\gamma,q\rangle+B_{\xi}(q)\geq0$, we have\\
$$
\left\{\begin{array}{lll}
Dim_{\nu}({\overline{\mathcal{K}}}^{\gamma})\leq\langle\gamma,q\rangle+B_{\xi}(q)\;\hbox{for}\;q\in\mathbb{R}_{+}^{k},\\
\\
Dim_{\nu}({\underline{K}}_{\gamma})\leq\langle\gamma,q\rangle+B_{\xi}(q)\;\hbox{for}\;q\in\mathbb{R}_{-}^{k}.
\end{array}
\right.	
$$
In particular, $Dim_{\nu}({\overline{\mathcal{K}}}^{\gamma})\leq\gamma_{i}$, $\forall1\leq i\leq k.$\\
3. Whenever $A\subseteq{\underline{K}}_{\gamma}$ is Borel with $\mu_{i}(A)>0$ for all $i$, we have $\gamma_{i}\leq\dim_{\nu}({A}).$\\
4. Whenever $A\subseteq{\underline{K}}_{\gamma}$ is Borel with $\mu_{i}(A)>0$ for all $i$, we have $\gamma_{i}\leq Dim_{\nu}({A}).$
\end{corollary}
Corollary \ref{corollaire-1-spectre} allows as to study the eventual link between the dimension functions $b_{\xi}$ and $B_{\xi}$ and the joint multifractal spectrum. Indeed, consider the following quantities,
$$
{\underline{a}}_{\xi}=\underset{q_{i}>0}{\sup}\left(-\frac{b_{\xi}(q)}{\left|q\right|}\right),\quad{\overline{a}}_{\xi}=\underset{q_{i}<0}{\inf}\left(-\frac{b_{\xi}(q)}{\left|q\right|}\right),
$$ 
$$
{\underline{A}}_{\xi}=\underset{q_{i}>0}{\sup}\left(-\frac{B_{\xi}(q)}{\left|q\right|}\right)\;\;\hbox{and}\;\;\overline{A}_{\xi}=\underset{q_{i}<0}{\inf}\left(-\frac{B_{\xi}(q)}{\left|q\right|}\right).
$$
We now establish upper bounds for $b_{\xi}$ and $B_{\xi}$. 
\begin{theorem}\label{upperbounds}
Consider a metric space $X$, a vector-valued Borel probability measure $\xi=(\mu,\nu)=(\mu_{1},\mu_{2},...,\mu_{k},\nu)$ on $\mathbb{R}^{d}$, and $\gamma=(\gamma_{1},\gamma_{2},...,\gamma_{k})\in\mathbb{R}^{k}_+$. It holds that
$$
dim_{\nu}(K(\gamma))\leq\left\{
\begin{array}{lll}
b_{\xi}^{\ast}(\gamma)=\underset{q}{\inf}(\langle\gamma,q\rangle+b_{\xi}(q))&,&\forall\gamma\in({\underline{a}}_{\xi},{\overline{a}}_{\xi})^k,\\ 
0&,&\forall\gamma\notin({\underline{a}}_{\xi},{\overline{a}}_{\xi})^k,
\end{array}
\right. 
$$
and
$$
Dim_{\nu}(K(\gamma))\leq\left\{
\begin{array}{lll}
B_{\xi}^{\ast}(\gamma)=\underset{q}{\inf}(\langle\gamma,q\rangle+B_{\xi}(q))&,&\forall\gamma\in({\underline{a}}_{\xi},{\overline{a}}_{\xi})^k,\\ 
0&,&\forall\gamma\notin({\underline{a}}_{\xi},{\overline{a}}_{\xi})^k.
\end{array}
\right. 
$$
\end{theorem}
\begin{proof} 
We will sketch the proof of the first part. The second may be checked by similar techniques. Since $\gamma_{i}\in({\underline{a}}_{\xi},{\overline{a}}_{\xi})$ then 
$K(\gamma)\neq\emptyset$. Consequently
$$
\langle\gamma,q\rangle+b_{\xi}(q)\geq0,
$$
which yields that
$$
dim_{\nu}(K(\gamma))\leq\left\{ 
\begin{array}{c}
\langle\gamma,q\rangle+b_{\xi}(q)\,;\;\forall q\in\mathbb{R}_{+}^{k},\\
\langle\gamma,q\rangle+b_{\xi}(q)\,;\;\forall  q\in\mathbb{R}_{-}^{k}.
\end{array}
\right.
$$
Consequently,
$$
dim_{\nu}(K(\gamma))\leq\underset{q}{\inf}(\langle\gamma,q\rangle+b_{\xi}(q))=b_{\xi}^{\ast}(q).
$$
\end{proof}
\section{Validity of an associated joint multifractal formalism}
In the classical case of single measures, the multifractal formalism is resumed in a mathematical formula stating that the fractal dimension of the singularities set in the Hausdorff sense (known as the spectrum of singularities) is evaluated by means of the Legendre transform of a free energy evaluated by some dimension functions $b$ and $B$, the original versions of $b_\xi$ and $B_\xi$ is our case. Such a formalism has been proved to hold for many cases of measures such as doubling, Gibbs, self similar, ... etc. Moreover, it has been extended for more large cases such as the multifractal case due to \cite{olsen1} and the mixed case due to \cite{Menceuretal}.

In the present part, we propose to develop an extending multifractal formalism relatively to the joint case introduced in the previous sections. Such a joint multifractal formalism conjectures that
$$
d(\gamma)=b_\xi^*(\gamma)=B_\xi^*(\gamma)
$$
for suitable $\gamma$. So, let
$$
E_{\xi}(\gamma)=\left\{x\in{S}_(\xi)\,;\;\hbox{ }\underset{r\rightarrow0}{\lim}\frac{\log(\mu(B(x,r))}{\log(\nu(B(x,r))}=\gamma\right\} ,
$$
where
$$
\dfrac{\log(\mu(B(x,r))}{\log(\nu(B(x,r))}=\Biggl(\dfrac{\log(\mu_1(B(x,r))}{\log(\nu(B(x,r))},\dfrac{\log(\mu_2(B(x,r))}{\log(\nu(B(x,r))},\dots,\dfrac{\log(\mu_k(B(x,r))}{\log(\nu(B(x,r))}\Biggr).
$$
The following theorem provides a case of validity of the extended joint multifractal formalism. 
\begin{theorem}\label{Formalisme-1}
Consider a vector-valued Borel probability measure $\xi=(\mu,\nu)=(\mu_{1},\mu_{2},...,\mu_{k})$ on $\mathbb{R}^{d}$ with compact support. Let $q\in\mathbb{R}^{k}$ at which $B_{\xi}$ satisfies $B_{\xi}^{q,B_{\xi}(q)}({S}_\xi)>0$. Then,
$$
Dim_{\nu}\left[E_{\xi}(-B_{\xi}^{\prime}(q))\right]=\dim_{\nu}\left[E_{\xi}(-B_{\xi}^{\prime}(q))\right]=B_{\xi}^{\ast}(-B_{\xi}^{\prime}(q)).
$$
where $f^{\ast}(\alpha)=\underset{\beta}{\inf}(\left\langle\alpha,\beta\right\rangle+f(\beta)).$
\end{theorem}
The proof reposes on the following lemmas.
\begin{lemma}\label{lemme-1-formalisme}
Let $\alpha =-\nabla B_{\xi}(q).$ Then
$$
\mathcal{H}_{\nu}^{\left\langle\alpha,q\right\rangle+B_{\xi}(q)-k\eta}(E_{\xi}(\alpha))\geq\mathcal{H}_{\xi}^{q,B_{\xi}(q)}(E_{\xi}(\alpha)),\;\forall\,\eta>0.
$$
\end{lemma}
\begin{lemma}\label{lemme-2-formalisme}
We have 
$$
\mathcal{H}_{\xi}^{q,B_{\xi}(q)}((S_{\mu}\cap S_{\nu})\diagdown E_{\xi}(-\nabla B_{\xi}(q)))=0.
$$
\end{lemma}
\hskip-17pt\textbf{Proof of Theorem \ref{Formalisme-1}.} By Lemma \ref{lemme-1-formalisme} and Lemma \ref{lemme-2-formalisme}, it holds that
$$
\mathcal{H}_{\nu}^{\left\langle\alpha,q\right\rangle+B_{\xi}(q)-k\eta}(E_{\xi}(\alpha))\geq\mathcal{H}_{\xi}^{q,B_{\xi}(q)}(E_{\xi}(\alpha))=0,\;\forall\,\eta>0.
$$
Hence,
$$
dim_{\nu}E_{\xi}(-\nabla B_{\xi}(q))\geq\left\langle-\nabla
B_{\xi}(q),q\right\rangle +B_{\xi}(q)-k\eta.
$$
Letting $\eta \rightarrow 0$, this yields that
$$
dim_{\nu}E_{\xi}(-\nabla B_{\xi}(q))\geq\left\langle-\nabla B_{\xi}(q),q\right\rangle+B_{\xi}(q),
$$
As a consequence of this result, we finish our paper by establishing an important result relating the multifractal formalism introduced here of the vector-valued measures to their projections according to some suitable subspaces of $\mathbb{R}^d$. 

Let $m\in\mathbb{N}$ be such that $0<m<d$, and $G_{d,m}$ be the Grassamannian manifold composed of all $m$-dimensional linear subspaces of $\mathbb{R}^{d}$. Let also $\gamma_{d,m}$ be the invariant Haar measure on $G_{d,m}$ satisfying $\gamma_{d,m}(G_{d,m})=1$. For $V\in G_{d,m}$, we designate by $P_V$ the usual orthogonal projection onto $V$. Consider also a Borel probability measure $\mu$ on $\mathbb{R}^{d}$, and denote $\mu_{V}$ its projection on $V$ relatively to $P_V$, i.e.,
$$
\mu_{V}(A)=\mu(P_{V}^{-1}(A)),\hbox{ }\forall A\subset V.
$$
The following result holds. It extends the results of \cite{Bahroun}, \cite{Barral-Bhouri}, \cite{Bmabrouk3}, \cite{Douzi-Selmi}, \cite{Falconer1}, \cite{Falconer2}, \cite{Hunt}, \cite{Kaufman}, \cite{Mattila2}, \cite{MattilaSaaramen}, \cite{ONeil}, \cite{Selmi1}, \cite{Selmi1}, \cite{Shmerkin}. 
\begin{theorem}\label{projectionresult}
Let $\xi=(\mu,\nu)=(\mu_{1},\mu_{2},...,\mu_{k})$ be a vector-valued Borel probability measure on $\mathbb{R}^{d}$ with $S_{\mu}=S_{\nu}$ being compact, and $q\in\mathbb{R}_{-}^{k}$. Assume further that
\begin{description}
\item[$(H_{1})$] $\mathcal{H}_{\xi}^{q,B_{\xi}(q)}(S_{\mu})>0$.
\item[$(H_{2})$] $B_{\xi}(q)<1$.
\item[$(H_{3})$] $B_{\xi}(q)$ is differentiable at $q$.
\end{description}
It holds for all $V\in G_{n,m}$ that
$$
\begin{array}{lll}
\dim_{\nu_{v}}\left(E_{\mu_{v},\nu_{v}}(-\nabla B_{\xi}(q))\right)&=&Dim_{\nu_{v}}\left(E_{\xi_{v}}(-\nabla B_{\xi}(q))\right)\\
&=&\dim_{\nu}(E_{\xi_{v}}(-\nabla B_{\xi}(q)))\\
&=&Dim_{\nu}(E_{\xi_{v}}(-\nabla B_{\xi}(q)))\\
&=&B^{\ast}(-\nabla B_{\xi}(q)).
\end{array}
$$
\end{theorem}
\begin{proof} 
Using equations (\ref{coupures-2a})-(\ref{coupures-2b}), Proposition \ref{existencecoupures} and the hypotheses $(H_{1})$ and $(H_{2})$, we get
\begin{equation}\label{5.3}
b_{\xi}(q)=B_{\xi}(q).
\end{equation}
Similarly, 
$$
0<\mathcal{H}_{\xi}^{q,B_{\xi}(q)}(S_{\mu})\leq\mathcal{H}_{\xi_{v}}^{q,B_{\xi_{v}}(q)}(S_{\mu_{v}}),\hbox{ }\forall\hbox{ }V\in G_{n,m}.
$$
Next, Theorem \ref{Formalisme-1} and equation (\ref{5.3}) imply that
$$
\dim_{\nu_{v}}\left[E_{\xi_{v}}(-\nabla B_{\xi}(q))\right]\geq\hbox{}\left\langle-\nabla B_{\xi}(q),q\right\rangle+B_{\xi}(q).
$$
The other estimation is satisfied since,
$$
\dim_{\nu_{v}}(E_{\xi_{v}}(-\nabla B_{\xi}(q)))\geq
\left\langle-\nabla B_{\xi}(q),q\right\rangle+B_{\xi_{v}}(q)=\left\langle-\nabla B_{\xi}(q),q\right\rangle +B_{\xi}(q).
$$
This achieves the proof.
\end{proof}
\section{Conclusion}
In the present paper, a class of quasi Ahlfors vector-valued measures has been considered for mixed multifractal analysis. It is shown that under the weak assumption of quasi Ahlfors, the construction of mixed generalizations of fractal measures such as Hausdorff and packing is possible. Such mixed multifractal analysis has induced in a natural way some corresponding multifractal dimensions, which have been shown to satisfy the multifractal formalism in mixed multifractal framework. The results developed here are subject of several extensions. Indeed, the developments showed that to get a valid variant of the multifractal formalism does not necessitate to apply radius power-laws equivalent measures. This leads to think about a general framework where the restriction of the vector-valued measure on balls $\mu(B(x,r))$ may be any vector-valued function $\varphi(r)$ which is not equivalent to power-laws $r^\alpha$ and develop a general formulation for mixed multifractal analysis. From the practical point of view, this is important, as it may lead for example to a mixed multifractal fluctuation analysis version where correlations, cross-correlations and auto-correlations have more flexible laws. Some initial thoughts may be extracted from \cite{Abry,Avishek,BenMabroukDFA,Biswas,Calvet-Fisher,Dai,Kinnison,Liu,Liu1}. 


\begin{thebibliography}{99}
\bibitem{Abry} P. Abry, H. Wendt and G. Didier, Detecting and estimating multivariate self-similar sources in high-dimensional noisy mixtures. 2018 IEEE Workshop on Statistical Signal Processing (SSP), 2017, pp. 688-692.

\bibitem{Attiaetal} N. Attia, B. Selmi and C. Souissi, Some density results of relative multifractal analysis. Chaos, Solitons and Fractals 103 (2017), 1--11.

\bibitem{Attia-Selmi1} N. Attia and B. Selmi, Relative Multifractal Box-Dimensions. Filomat 33:9 (2019), 2841–2859. DOI: 10.2298/FIL1909841A.
	
\bibitem{Aversa-Bandt}	V. Aversa and  C. Bandt, The Multifractal Spectrum of Discrete Measures. Acta Univ. Carolinae-Math. Et Phys. 31(2) (1990), pp. 5-8.

\bibitem{Avishek} B. Avishek, Long memory and fractality among global equity markets: A multivariate wavelet approach. MPRA Paper No. 99653, posted 20 Apr 2020, 24 pages, https://mpra.ub.uni-muenchen.de/99653.

\bibitem{Bahroun} F. Bahroun and I. Bhouri, Multifractals and projections. Extr. Math., 2006, vol. 21, pp. 83--91.
		
\bibitem{Barlow-Taylor} M. T. Barlow and S. J. Taylor, Defining fractal subsets of $\mathbb{Z}^d$. Proc. London Math. Soc. 64(3) (1992), pp. 125-152.
	
\bibitem{Barral-Bhouri} J. Barral and I. Bhouri, Multifractal analysis for projections of gibbs and related measures. Ergod Theory Dyn Syst. 2011;31:673–701.
	
\bibitem{BarreiraSaussol} L. Barreira and B. Saussol, Variational principles and mixed multifractal spectra. Transactions of the American Mathematical Society 353(10) (2001), 3919--3944.
	
\bibitem{Bmabrouk1}	A. Ben Mabrouk, A note on Hausdorff and packing measures, Interna. J. Math. Sci., 8(3-4) (2009), 135-142.
	
\bibitem{Bmabrouk3}	A. Ben Mabrouk, A higher order multifractal formalism, Stat. Prob. Lett. 78 (2008), pp.~1412-1421.

\bibitem{BenMabroukDFA} A. Ben Mabrouk, Fluctuation analysis, Detrended fluctuation analysis, Fractal detrended analysis, Wavelet detrended fluctuation analysis: Extension to Mixed Cases. Project: Selected Topics on Wavelet Analysis and Statistical Applications. May 2018, DOI: 10.13140/RG.2.2.32902.27209.

\bibitem{Billingsley} P. Billingsley, Ergodic Theory and Information. John Wiley \& Sons, Inc., New York, London, Sydney 1965.

\bibitem{Biswas} A. Biswas, H. P. Cresswell and C. S. Bing, Application of Multifractal and Joint Multifractal Analysis in Examining Soil Spatial Variation: A Review. Chapter 6 in Fractal Analysis and Chaos in Geosciences. Edited by Sid-Ali Ouadfeul, InTechOpen, 2012, pp. 109-138. DOI.org/10.5772/51437.

\bibitem{Bozkus} S. K. Bozkus, H. Kahyaoglu and A. M. M. Lawali, Multifractal analysis of atmospheric carbon emissions and OECD industrial production index. International Journal of Climate Change Strategies and Management Vol. 12 No. 4, 2020 pp. 411-430. DOI: 10.1108/IJCCSM-08-2019-0050.	

\bibitem{Calvet-Fisher} L. Calvet and A. Fisher, Multifractal volatility, theory, forecasting, and pricing. Academic press advanced finance series, 1st ed, September 2008.

\bibitem{Castiglioni-Faini} P. Castiglioni and A. Faini, A Fast DFA Algorithm for Multifractal Multiscale Analysis of Physiological Time Series. Front. Physiol. 10 (2019), 18 pages, Article: 115, DOI: 10.3389/fphys.2019.00115

\bibitem{Cattani} C. Cattani and J. Rushchitsky, Wavelet and Wave Analysis as applied to Materials with Micro or Nanostructure. World Scientific Publishing Company, 2007.

\bibitem{Chu} P. C. Chu, Multifractal analysis of the southwestern iceland sea surface mixed layer thermal structure. Thirteenth Symposium on Boundary Layers and Turbulence, American Meteorological Society, 476-479, 1999. http://hdl.handle.net/10945/36212.

\bibitem{Cole} J. Cole, Relative multifractal analysis. Choas Solitons Fractals 11 (2000), pp. 2233-2250.  
	
\bibitem{Das} M. Das, Hausdorff measures, dimensions and mutual singularity. Trans Am Math Soc 357 (2005), pp. 4249-4268 .

\bibitem{Dai} M. Dai, J. Hou, J. Gao, W. Su, L. Xi and D. Ye, Mixed multifractal analysis of China and US stock index series. Chaos, Solitons and Fractals 87 (2016) 268–275.

\bibitem{Dai1} M. Dai, S. Shao, J. Gao, Y. SUN and W. SU, Mixed multifractal analysis of crude oil, gold and exchange rate series. Fractals, 24(4) (2016) 1650046 (7 pages), DOI: 10.1142/S0218348X16500468.

\bibitem{Douzi-Selmi} Z. Douzi and B. Selmi, Multifractal variation for projections of measures. Chaos Solitons Fractals 91 (2016), pp. 414-420 .

\bibitem{DouziERA} Z. Douzi and B. Selmi, On the mutual singularity of multifractal measures. Electronic Research Archive. volume 28 number 1, march 2020 pages 423-432. DOI: 10.3934/era.2020024

\bibitem{DouziChaos2019} Z. Douzi and B. Selmi, Regularities of general Hausdorffand packing functions. Chaos, Solitons and Fractals 123 (2019) 240–243.

\bibitem{DouziATA} Z. Douzi and B. Selmi, On the Projections of the Mutual Multifractal R\'enyi Dimensions. Anal. Theory Appl., 36(2) (2020), pp. 1-20. DOI: 10.4208/ata.OA-2017-0036.

\bibitem{DouziSamtiSelmi} Z. Douzi, A. Samti and B. Selmi, Another example of the mutual singularity of multifractal measures. Proyeccinones (Antofagasta. On line) 40(1) (2021), pp. 17-23.

\bibitem{Drozdz} S. Drozdz, R. Kowalski, P. Oswiecimk, R. Rak and R. Gebarowski, Dynamical Variety of Shapes in Financial Multifractality. Complexity, Volume 2018, Article ID 7015721, 13 pages. DOI: 10.1155/2018/7015721.

\bibitem{Edgar} G. A. Edgar, Centered densities and fractal measures. New York J Math 13 (2007), pp. 33-87.
	
\bibitem{Falconer1} K. J. Falconer and J. D. Howroyd, Packing dimensions of projections and dimensions profiles. Math Proc Cambridge Philos Soc 121 (1997), 269--86.
	
\bibitem{Falconer2} K. J. Falconer and P. Mattila, The packing dimensions of projections and sections of measures. Math Proc Cambridge Philos Soc 119 (1996), 695--713.

\bibitem{Fan} Q. Fan, S. Liu, K. Wang, Multiscale multifractal detrended fluctuation analysis of multivariate time series, Physica A: Statistical Mechanics and its Applications, 532 (2019), 121864, DOI: 10.1016/j.physa.2019.121864.

\bibitem{Ganetal} G. Gan, C. Ma and J. Wu, Data Clustering Theory, Algorithms, and Applications. ASA-SIAM Series on Statistics and Applied Probability, SIAM, Philadelphia, ASA, Alexandria, VA, 2007.
	
\bibitem{Glasscock} D. Glasscock, Marstrand-type Theorems for the Counting and Mass Dimensions in $\mathbb{Z}^d$. Combinatorics, Probability and Computing 25 (2016), pp. 700-743.

\bibitem{Hunt} B. R. Hunt and V. Y. Kaloshin, How projections affect the dimension spectrum of fractal measures. Nonlinearity 1997;10:1031–46.
	
\bibitem{Jarvenpaaetal} E. Jarvenpaa, M. Jarvenpaa, A. Kaenmaki, T. Rajala, S. Rogovin and V. Suomala, Packing Dimension and Ahlfors Regularity of Porus Sets in Metric Spaces. arXiv:1701.08593v1 [math.CA] 30 Jan 2017
	
\bibitem{Jiang} Z.-Q. Jiang, W.-J. Xie, W.-X. Zhou and D. Sornette, Multifractal analysis of financial markets: a review. Reports on Progress in Physics, Volume 82, Number 12, ID 125901, 105 pages.



\bibitem{Kaufman} R. Kaufman, On Hausdorff dimension of projections. Mathematika 15 (1968), pp. 153--5.

\bibitem{Kinnison} A. Kinnison and P. Mörters, Simultaneous Multifractal Analysis of the Branching and Visibility Measure on a Galton-Watson Tree. Advances in Applied Probability, 42(1) (2010), 226-245. DOI:10.1239/aap/1269611151.

\bibitem{Liu} R. Liu and T. Lux, Non-homogeneous volatility correlations in the bivariate multifractal model, The European Journal of Finance, 21(12) (2015), pp. 971-991, DOI: 10.1080/1351847X.2014.897960.

\bibitem{Liu1} R. Liu and T. Lux, Generalized Method of Moment estimation of multivariate multifractal models. Economic Modelling 67 (2017), pp. 136-148. DOI: 10.1016/j.econmod.2016.11.010.

\bibitem{Maneveauetal} C. Meneveau, K. R. Sreenivasan, P. Kailasnath, and M. S. Fan, Joint multifractal measures: Theory and applications to turbulence. Physical Review A 41(2) (1990), 894-913.

\bibitem{Manshour} P. Manshour, Nonlinear correlations in multifractals: Visibility graphs of magnitude and sign series. Chaos 30, 013151 (2020), 9 pages. DOI: 10.1063/1.5132614.

\bibitem{Mattila} P. Mattila, Geometry of sets and measures in euclidian spaces: fractals and rectifiability. Cambridge University Press; 1995 .
	
\bibitem{Mattila2} P. Mattila, Hausdorff dimension, orthogonal projections and intersections with planes. Annales academiae scientiarum fennicae. Series A I Mathematica, 1975, 1, 227--44.
	
\bibitem{MattilaSaaramen} P. Mattila and P. Saaramen, Ahlfors-David regular sets and bilipschitz maps. arXiv:0809.4877v1 [math.MG] 29 Sep 2008.
	
\bibitem{Menceuretal} M. Menceur, A. Ben Mabrouk and K. Betina, The Multifractal Formalism For Measures, Review and Extension to Mixed Cases. Anal. Theory Appl., Vol. 32, No. 1 (2016), pp. 77-106.
	
\bibitem{olsen1} L. Olsen, A multifractal formalism, Adv. Math., 116 (1995), 82-196.
	
\bibitem{olsen2} L. Olsen, Multifractal dimensions of product measures. Math Proc Camb Phil Soc 120 (1996), 709-34.
	
\bibitem{olsen2b} L. Olsen, Dimension inequalities of multifractal Hausdorff measures and multifractal packing measures, Math. Scand., 86 (2000), 109-129.
	
\bibitem{olsen2c} L. Olsen, Integral, probability, and fractal measures, by G. Edgar, Springer, New York, 1998, Bull. Amer. Math. Soc., 37 (2000), 481-498.
	
\bibitem{olsen2d} L. Olsen, Divergence points of deformed empirical measures, Math. Resear. Letters, 9 (2002), 701-713.
	
\bibitem{olsen2e} L. Olsen, Mixed divergence points of self-similar measures, Indiana Univ. Math. J., 52 (2003), 1343-1372.
	
\bibitem{olsen3} L. Olsen, Mixed generalized dimensions of self-similar measures, J. Math. Anal. and Appl., 306 (2005), 516-539.
	
\bibitem{ONeil} T. C. O'Neil, The multifractal spectra of projected measures in Euclidean spaces. Chaos, Solitons \& Fractals 2000;, 1, 901--21.

\bibitem{Oral} E. Oral and G. Unal, Modeling and forecasting time series of precious metals: a new approach to multifractal data. Financial Innovation (2019) 5:22. DOI: 10.1186/s40854-019-0135-3.

\bibitem{Pajot} H. Pajot, Sous-ensembles de courbes Ahlfors-r\'eguli\`eres et nombres de Jones. Publicacions Matematiques, 40 (1996), 497--526.
	
\bibitem{Pesin-Weiss} Y. B. Pesin, Dimension Theory in Dynamical Systems. Contemporary Views and Applications (Chicago Lectures in Mathematics) (Chicago, IL: University of Chicago Press), 1997.
	
\bibitem{Qu} C. Qu, The Lower Densities of Symmetric Perfect Sets, Anal. Theory Appl., 28 (2012), pp. 377-384.
	
\bibitem{Selmi1} B. Selmi, A note on the effect of projections on both measures and the generalization of $q$-dimension capacity. Probl. Anal. Issues Anal. Vol. 5 (23), No. 2, 2016, pp. 38--51.

\bibitem{Selmi2} B. Selmi, measure of relative multifractal exact dimensions. Advances and Applications in Mathematical Sciences 17(10) (2018), 629-643.

\bibitem{SelmiActa} B. Selmi, The relative multifractal analysis, review and examples. Acta Sci. Math. (Szeged) 86 (2020), 635–666. DOI: 10.14232/actasm-020-801-8. 

\bibitem{SelmiProy2021} B. Selmi, Remarks on the mutual singularity of multifractal measures. Proyeccinones (Antofagasta. On line) 40(1) (2021), pp. 73-84.

\bibitem{Shmerkin} P. Shmerkin, Projections of self-similar and related fractals: a survey of recent developments. 2015. ArXiv:1501.00875v1.

\bibitem{Wang} H-Y. Wang and Y.-S. Feng, Multivariate correlation analysis of agricultural futures and spot markets based on multifractal statistical
methods. J. Stat. Mech. (2020) 073403. DOI.org/10.1088/1742-5468/ab900f.

\bibitem{Wang1} J. Wang, W. Shao and J. Kim, Multifractal detrended cross-correlation analysis between respiratory diseases and haze in South Korea. Chaos, Solitons and Fractals 135 (2020) 109781, 10 pages.

\bibitem{Wendt} H. Wendt, S. Combrexelle, Y. Altmann, J.-Y. Tourneret, S. McLaughlin and P. Abry, Multifractal Analysis of Multivariate Images Using Gamma Markov Random Field Priors. SIAM J. Imaging Sciences Vol. 11, No. 2, (2018), pp. 1294-1316.

\bibitem{Xu-Wang} M. Xu and S. Wang, The Boundedness of Bilinear Singular Integral Operators on Sierpinski Gaskets, Anal. Theory Appl., 27 (2011), pp. 92-100.
	
\bibitem{Xu-Xu}	S. Xu and W. Xu, Note on the Paper "An Negative Answer to a Conjecture on the Self-similar Sets Satisfying the Open Set Condition", Anal. Theory Appl., 28 (2012), pp. 49-57.
	
\bibitem{Xu-Xu-Zhong} S. Xu, W. Xu and D. Zhong, Some New Iterated Function Systems Consisting of Generalized Contractive Mappings, Anal. Theory Appl., 28 (2012), pp. 269-277.
	
\bibitem{Ye1} Y.-L. Ye, Self-similar vector-valued measures Adv. Appl. Math. 38 (2007): 71..96.
	
\bibitem{Yuan} Y. Yuan, Spectral Self-Affine Measures on the Generalized Three Sierpinski Gasket, Anal. Theory Appl., 31 (2015), pp. 394-406.
	
\bibitem{Zeng-Yuan-Xui} C. Zeng, D. Yuan and S. Xui, The Hausdorff Measure of Sierpinski Carpets Basing on Regular Pentagon, Anal. Theory Appl., 28 (2012), pp. 27-37.
	
\bibitem{Zhou-Feng} Z. Zhou and L. Feng, A Theoretical Framework for the Calculation of Hausdorff Measure Self-similal Set Satisfying OSC, Anal. Theory Appl., 27 (2011), pp. 387-398.
	
\bibitem{Zhu-Zhou} Z. Zhu and Z. Zhou, A Local Property of Hausdorff Centered Measure of Self-Similar Sets, Anal. Theory Appl., 30 (2014), pp. 164-172.

\section*{Further readings}

\bibitem{Cao} G. Cao, L.-Y. He and J. Cao, Multifractal Detrended Analysis Method and Its Application in Financial Markets. Springer 2018.

\bibitem{Cattani2007} C. Cattani and J. Rushchitsky, Wavelet and Wave Analysis as applied to Materials with Micro or Nanostructure. World Scientific 2007.

\bibitem{Chandrasekhar} E. Chandrasekhar, V. P. Dimri and V. M. Gadre, Wavelets and Fractals in Earth System Sciences. Taylor \& Francis Group 2014.

\bibitem{Dauphine} A. Dauphin\'e, Fractal Geography. John Wiley \& Sons, Inc. 2012.

\bibitem{Ghanbarian} B. Ghanbarian and A. G. Hunt, Fractals Concepts and Applications in Geosciences. Taylor and Francis Group, 2017.

\bibitem{Ghosh} D. Ghosh, S. Samanta and S. Chakraborty, Multifractals and
Chronic Diseases of the Central Nervous System. Springer 2019.

\bibitem{Janahmadov} A. K. Janahmadov and M. Y. Javadov, Synergetics and Fractals in Tribology, Materials Forming, Machining and Tribology. Springer 2016.

\bibitem{Nakayama} T. Nakayama and K. Yakubo, Fractal Concepts in Condensed Matter Physics. Springer 2003.

\bibitem{Ouadfeul} S.-A. Ouadfeul, Fractal analysis and chaos in geosciences. InTech Open, 2012. 

\bibitem{Seuront} L. Seuront, Fractals and multifractals in ecology and aquatic science. Taylor and Francis Group, 2010.

\bibitem{Tolotti} M. Tolotti, Fractal and multifractal models for price changes, An attempt to manage the Black Swan. Corso di Laurea magistrale
in Economia e Finanza, Universita CaFoscari Venisia, 2012/2013.

\bibitem{ZhangBook} Z.-X. Zhang, Rock Fracture and Blasting Theory and Applications. Elsevier 2016.
\end{thebibliography}
\end{document}